\documentclass[11pt]{amsart}

\usepackage[left=1in, right=1in, top=1.2in, bottom=1.2in]{geometry}
\usepackage{microtype, mathtools, mathrsfs, enumerate, dsfont, esvect}
\usepackage{tikz,wasysym}
\usepackage{pgf}

\usetikzlibrary{positioning,automata}
\usepackage[font=small,labelfont=bf]{caption}
\usepackage{bbm}
\usepackage{verbatim}
\usepackage[linktocpage]{hyperref}
\hypersetup{
    colorlinks=true,        
    linkcolor=red,          
    citecolor=blue,         
    filecolor=magenta,      
    urlcolor=cyan           
}

\usepackage{amssymb}
\usepackage{url}

\newtheorem{theorem}{Theorem}[section]
\newtheorem{proposition}[theorem]{Proposition}

\newtheorem{lemma}[theorem]{Lemma}

\theoremstyle{definition}

\newtheorem{definition}[theorem]{Definition}

\newtheorem{notation}[theorem]{Notation}

\theoremstyle{plain}
\numberwithin{equation}{theorem}

\theoremstyle{remark}

\newtheorem{remark}[theorem]{Remark}

\DeclareMathOperator{\length}{length}

\newif\ifhascomments \hascommentstrue
\ifhascomments
  \newcommand{\dragos}[1]{{\color{red}[[\ensuremath{\bigstar\bigstar\bigstar} #1]]}}
  \newcommand{\matt}[1]{{\color{red}[[\ensuremath{\spadesuit\spadesuit\spadesuit} #1]]}}
\else
  \newcommand{\dragos}[1]{}
  \newcommand{\matt}[1]{}
\fi

\begin{document}

\title[A refinement of Christol's theorem]{A refinement of Christol's theorem for algebraic power series}

\author{Seda Albayrak}
\address{Department of Pure Mathematics\\
University of Waterloo\\
200 University Ave. W.
Waterloo, ON N2L 3G1\\
Canada}
\email{gsalbayr@uwaterloo.ca}
\author{Jason P. Bell}
\address{Department of Pure Mathematics\\
University of Waterloo\\
200 University Ave. W.
Waterloo, ON N2L 3G1\\
Canada}
\email{jpbell@uwaterloo.ca}

\begin{abstract} A famous result of Christol gives that a power series $F(t)=\sum_{n\ge 0} f(n)t^n$ with coefficients in a finite field $\mathbb{F}_q$ of characteristic $p$ is algebraic over the field of rational functions in $t$ if and only if there is a finite-state automaton accepting the base-$p$ digits of $n$ as input and giving $f(n)$ as output for every $n\ge 0$. An extension of Christol's theorem, giving a complete description of the algebraic closure of $\mathbb{F}_q(t)$, was later given by Kedlaya.  When one looks at the support of an algebraic power series, that is the set of $n$ for which $f(n)\neq 0$, a well-known dichotomy for sets generated by finite-state automata shows that the support set is either sparse---with the number of $n\le x$ for which $f(n)\neq 0$ bounded by a polynomial in $\log(x)$---or it is reasonably large in the sense that the number of $n\le x$ with $f(n)\neq 0$ grows faster than $x^{\alpha}$ for some positive $\alpha$. The collection of algebraic power series with sparse supports forms a ring and we give a purely algebraic characterization of this ring in terms of Artin-Schreier extensions and we extend this to the context of Kedlaya's work on generalized power series. \end{abstract}

\subjclass[2010]{Primary: 11B85; Secondary: 12J25, 13J05}

\keywords{Christol's theorem, finite-state automata, algebraic power series, automatic sequences, unramified extensions}

\thanks{The second author was supported by a Discovery Grant from the National Sciences and Engineering Research Council of Canada.}

\maketitle

\section{Introduction}
One of the fundamental results in the theory of finite-state automata is Christol's theorem \cite[Theorem 12.2.5]{AS} (see also \cite{C1, C2}), which asserts that a power series $F(t)=\sum a_n t^n$ with coefficients in a finite field $\mathbb{F}_q$ is algebraic over the field of rational functions $\mathbb{F}_q(x)$ if and only if the sequence $\{a_n\}$ is $p$-automatic.  More precisely, there is a finite-state machine accepting the base-$p$ expansion of $n$ as input, reading right-to-left, and giving $a_n\in \mathbb{F}_q$ as output
(see \S\ref{prelim} for details).

When one adopts this automaton-theoretic point of view, it is natural to ask how properties of automatic sequences and algebraic properties of power series correspond. An important dichotomy that arises in the theory of automatic sequences is the sparse/non-sparse partition of such sequences. Given a $p$-automatic sequence $a_n$ taking values in a ring, we can construct its support set $S=\{n\in \mathbb{N}\colon a_n\neq 0\}$. Such sets are called $p$-automatic sets and there is a striking gap in possible growth types of these sets: $\pi_S(x)=|\{n\le x\colon n\in S\}|$ is either poly-logarithmically bounded (i.e., $\pi_S(x)={\rm O}((\log\,x)^d)$ for some $d\ge 1$) or it grows at least like a fractional power of $x$ (i.e., there exists $\alpha>0$ such that $\pi_S(x)\ge x^{\alpha}$ for $x$ sufficiently large).  We call automatic sets $S$ for which the growth function is poly-logarithmically bounded \emph{sparse}, and we call an automatic sequence \emph{sparse} if it has a sparse support set. Such sets arise in many important contexts in which automatic sequences naturally appear.  We give a few such examples now. Derksen's positive characteristic version of the Skolem-Mahler-Lech theorem \cite{Derksen} (see also \cite[Chap. 11]{BGT}) shows that the zero set of a linearly recurrent sequence over a field of characteristic $p>0$ is a finite union of arithmetic progressions augmented by a sparse $p$-automatic set; Kedlaya's \cite{K, K2} work on extending Christol's theorem to give a full characterization of the algebraic closure of $\mathbb{F}_p(t)$ works by generalizing the notion of automatic sequences to maps $f:S_p\to \mathbb{F}_q$, where $S_p$ is the set of nonnegative elements of $\mathbb{Z}[p^{-1}]$, and as part of his work, he shows that for the maps that arise, the post-radix point behaviour of the support of $f$ can be described in terms of sparse automatic sequences.  Moosa and Scanlon's \cite{MS} (see also \cite{Ghioca}) work on the isotrivial case of the Mordell-Lang conjecture deals with $F$-sets and these can again be described using sparse automatic sets (see, also, \cite{BM}). Recent work of the second-named author with Hare and Shallit \cite{BHS} shows that automatic sets that are additive bases of the natural numbers can be characterized in terms of this sparse property as well.

In light of the importance of the sparse/non-sparse dichotomy for automatic sets, it is natural to ask for a characterization of the algebraic power series over a finite field with sparse supports.  The goal of this paper is to give such a characterization.  As it turns out, sparse series are intimately linked with Artin-Schreier extensions, which are the main reason for the fact that the algebraic closure of the field of Laurent series $K((t))$ over a field $K$ is considerably simpler when $K$ is of characteristic zero than when it is of positive characteristic.  Algebraic power series with sparse supports form a ring, and it is natural to ask whether this subalgebra of the ring of algebraic power series can be characterized purely algebraically.  In fact, we are able to give two characterizations.  The first characterization involves being the smallest non-trivial subalgebra that is closed under certain natural sparse-preserving operators.  Given an algebraic power series $F(t)$ with $F(0)=0$, there are four natural ways to produce additional sparse series: the first, is to apply the Artin-Schreier operator $F\mapsto F+F^p +F^{p^2}+\cdots$; the second, is to make a change of variables $t\mapsto \alpha t$ with $\alpha$ a nonzero scalar; the third, is to make a substitution of the form $F(t)\mapsto F(t^c)$ with $c$ a suitable positive rational number; and the fourth, is to multiply or divide by a power of $t$, assuming that one remains in the ring of formal power series.  As it turns out, we show that all sparse series can be built from these simple operations.  Using this, we are able to give a second, purely algebraic characterization, in terms of integral elements in a certain maximal unramified extension (see \S\ref{ur} for background).  More precisely, our main result is the following theorem.

\begin{theorem}\label{thm:main} Let $p$ be a prime and let $A$ be the field extension of $\bar{\mathbb{F}}_p(t)$ consisting of algebraic Laurent series over $\bar{\mathbb{F}}_p$, and let $B$ be the subring of $A$ consisting of algebraic power series with sparse support.  Then we have the following:
\begin{enumerate}
\item[(a)] $B$ is the smallest non-trivial $\bar{\mathbb{F}}_p$-subalgebra of $A$ that possesses the following closure properties:
\begin{enumerate}
\item[(P1)] If $F(t)\in B$ and $F(0)=0$ then $F(t)+F(t^p)+F(t^{p^2})+\cdots \in B$;
\item[(P2)] If $F(t)\in B$ and $\alpha\in \bar{\mathbb{F}}_p$ then $F(\alpha t)\in B$;
\item[(P3)] if $F(t) \in B$ and $t^d F(t^c) \in A$ with $c\in \mathbb{Q}_{>0}$ and $d\in \mathbb{Q}$, then $t^d F(t^c)\in B$;
\end{enumerate}
\item[(b)] If $F(t)$ is a power series with coefficients in $\bar{\mathbb{F}}_p$ then $F(t)\in B$ if and only if there is some $j\ge 0$ such that for $G(t):=F(t^{p^j})$ the following hold:
\begin{enumerate}
\item[(i)] the Galois closure of $G(t)$ over the field $K:=\bar{\mathbb{F}}_p(t^{\pm 1/n}\colon n\ge 1,p\nmid n)$ has degree a power of $p$;
\item[(ii)] the extension $K(G(t))/K$ is unramified outside of $0$ and $\infty$;
\item[(iii)] $G(t)$ is integral over the Laurent polynomial ring $\bar{\mathbb{F}}_p[t^{\pm 1}]$.
\end{enumerate}
\end{enumerate}
\end{theorem}
We point out that in item (ii) in part (b) of Theorem \ref{thm:main}, we are implicitly using the fact that places of $K$ are parametrized by points in $\mathbb{P}_{\bar{\mathbb{F}}_p}^1$ (see \S \ref{mainresb} for details) and so it is this identification of places with points in projective space that is being used when we speak of unramified extensions outside of $0$ and $\infty$.  

Christol's theorem has since been extended to numerous settings including power series over general base fields of positive characteristic \cite{SW}, multivariate power series \cite{S1, S2}. Kedlaya \cite{K, K2} gave an extension of Christol's theorem, which gives a complete characterization of the algebraic closure of $\mathbb{F}_p(t)$ in terms of finite-state automata and so-called generalized Laurent power series (see \S\ref{prelim} for details).  We in fact prove a version of Theorem \ref{thm:main} for generalized power series (see Theorem \ref{thm:main2}), which is in a sense more natural and from which one can deduce much of Theorem \ref{thm:main}, with a small amount of additional work.

The outline of the paper is as follows.  In \S\ref{prelimi} we give some background on the theory of finite-state machines and $p$-automatic sets and sequences as well as the basic number theoretic background on unramified extensions and Artin-Schreier extensions of positive characteristic fields.  In \S\ref{sec:sparse}, we define sparse languages and sets and give decomposition results for such sets that we use in the proof of part (a) of Theorem \ref{thm:main}. In \S\ref{mainres} and \S\ref{mainresb} we prove our main result along with a more general result characterizing generalized Laurent series with sparse support, which arise in Kedlaya's extension of Christol's theorem. 
\section{Preliminaries} \label{prelimi}
In this section, we give some of the basic background necessary for this paper. Since the paper bridges two largely disconnected areas of mathematics---namely, automata theory and algebraic number theory---we give an overview of the required concepts for the reader's convenience.

\subsection{Finite-state Automata}
\label{prelim}
In this subsection, we give some definitions and elementary facts concerning finite-state automata, regular languages, and $p$-automatic sequences and sets. 

Let $\Sigma$ be a nonempty finite set. We call $\Sigma$ an \emph{alphabet} and we call a finite or infinite sequence of symbols chosen from $\Sigma$ a \emph{word} over the alphabet $\Sigma$. We let $\Sigma^*$ denote the set of all finite words over $\Sigma$; that is, $\Sigma^*$ is the free monoid on the set $\Sigma$, with multiplication given by concatenation. In the case when $\Sigma=\{x\}$ is a singleton, we will use $x^*$ to denote $\Sigma^*$.  

Formally, a \emph{deterministic finite automaton with output} (DFAO) is a $6$-tuple $$M = (Q, \Sigma, \delta, q_0, \Delta, \tau),$$ where $Q$ is a finite set of states, $\Sigma$ is a finite input alphabet, $\delta$ is the transition function from $\Sigma\times Q$ to $Q$, $q_0 \in Q$ is the initial state, $\Delta$ is an output alphabet, and $\tau$ is the output function from $Q$ to $\Delta$.  Intuitively, we can think of a DFAO as a directed graph in which the vertices are the elements of $Q$ and for each vertex $q\in Q$ and each $s\in \Sigma$ we have a directed arrow with label $s$ from $q$ to the state $q'=\delta(s,q)$.  Then given a word $w\in \Sigma^*$, the DFAO gives us an output in $\Delta$ as follows: we begin at the initial state $q_0$ and then, reading $w$ from right to left, we obtain a path by moving vertex to vertex as we read each letter in $w$.  When we have finished reading $w$, we arrive at a state $q\in Q$ and we then apply $\tau$ to obtain an output in $\Delta$.  Adopting this point of view, we see we can extend $\delta$ to a map from $\Sigma^*\times Q$ to $Q$ and then the output associated to a word $w\in \Sigma^*$ is simply $\tau(\delta(w,q_0))$.  Thus to a DFAO $M$, there is an associated finite-state function $f_M:\Sigma^*\to \Delta$ given by $f_M(w)=\tau(\delta(w,q_0))$.

To aid with the reader's understanding, we give a concrete example of a DFAO in Figure 1. This DFAO has two states $q_0$ and $q_1$ with initial state $q_0$.  The alphabet $\Sigma$ is $\{0,1\}$ and the arrows in the graph indicate that the transitions are given by $\delta(0,q_0)=q_0$, $\delta(0,q_1)=q_1$, $\delta(1,q_0)=q_1$, and $\delta(1,q_1)=q_0$. The labelling $q_0/0$ and $q_1/1$ inside states is specifying that $\tau(q_0)=0$ and $\tau(q_1)=1$ (i.e. in the state label $q/a$, $a$ is the output associated with the state $q$).  If we regard a word over the alphabet $\{0,1\}$ with no leading zeros as being the binary expansion of a natural number then we can construct a $\{0,1\}$-valued sequence $f(n)$ by feeding the binary expansion of $n$ into our DFAO, starting at state $q_0$ and reading the word from right to left, and then applying $\tau$ to the value of the state we reach after we have fed all of the digits into the machine.  For example, if $n=13$, we have the binary expansion $1101$, and applying successive transitions we see that $1101$ takes state $q_0$ to state $q_1$ and so $f(13)=\tau(q_1)=1$.  The sequence obtained in this particular case via this procedure is known as the Thue-Morse sequence. 

\begin{figure}[!htbp]
\begin{tikzpicture}[shorten >=1pt,node distance=2cm,on grid]
  \node[state,initial]   (q_0)                {$q_0/0$};
  \node[state]           (q_1) [right=of q_0] {$q_1/1$};
  \path[->] (q_0) edge [loop above]   node         {0} ()
		 	edge  [bend left]   	node [above] {1} (q_1)
		(q_1) edge  [loop right]  node {0} ()
			edge [bend left] node [below] {1} (q_0);
  \end{tikzpicture}
  \caption{The DFAO generating the Thue-Morse sequence.}
\end{figure}
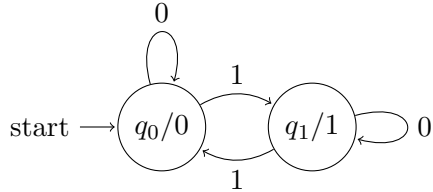
  
A language over $\Sigma$ is a subset of $\Sigma^*$.  We say that a language $\mathcal{L}\subseteq \Sigma^*$ is a \emph{regular language} if there is a DFAO $M=(Q, \Sigma, \delta, q_0, \Delta, \tau)$ with $\Delta=\{0,1\}$ such that $\mathcal{L}$ is precisely the set of words $w\in\Sigma^*$ for which $f_M(w)=1$.  In this case, we are really thinking of the automaton as accepting words $w$ with $f_M(w)=1$ and as rejecting words $w$ with $f_M(w)=0$, and so a regular language is the collection of words that are accepted by some DFAO.

 Let $k\ge 2$ be a natural number and let $\Sigma_k$ be the alphabet $\{0,1,\ldots ,k-1\}$.  Then for every natural number $n$, there is a word $w=(n)_k\in \Sigma_k^*$, which is the base-$k$ expansion of $n$, where we define $(0)_k$ to be the empty word; conversely, given a non-empty word $w\in \Sigma_k^*$ with no leading zeros there is a natural number $n=[w]_k$, which is the natural number whose base-$k$ expansion is $w$. In the case when $w$ is the empty word, we take $[w]_k=0$. A sequence $a: \mathbb{N} \rightarrow \Delta$ is called \emph{$k$-automatic} if there exists a DFAO $M=(Q, \Sigma_k, \delta, q_0, \Delta, \tau)$ such that for each $n \in \mathbb{N}$, $a(n)$ can be computed as  $f_M\left( (n)_k \right)$. We then say that a subset $S \subseteq \mathbb{N}$ is a \emph{$k$-automatic set} if the characteristic function of $S$, $\chi_S:\mathbb{N}\to \{0,1\}$ defines a $k$-automatic sequence. 

Christol's theorem is a fundamental result that gives a characterization of the collection of algebraic power series with coefficients in a finite field in terms of automatic sequences. 
We recall that given a field $\mathbbm{k}$, a power series $F(t)\in \mathbbm{k}[[t]]$ is an \emph{algebraic power series} if there exists some $s\ge 1$ and $B_0(t),B_1(t), \dots, B_s(t)\in \mathbbm{k}[t]$, not all zero, such that $$B_s(t) F(t)^s + \cdots + B_1(t) F(t) + B_0(t)=0.$$  Equivalently, $F(t)$ is a power series that is algebraic over the field of rational functions $\mathbbm{k}(t)$. The collection of algebraic power series forms a ring and it contains the power series expansions of rational functions that are regular at $t=0$.  We now state the famous result of Christol \cite[Theorem 12.2.5]{AS}.

\begin{theorem}[Christol] Let $p$ be a prime number, let $q$ be a power of $p$, and let $F(t) = \sum_{n \geq 0} f(n) t^n \in \mathbb{F}_q [[t]]$. Then $F(t)$ is an algebraic power series if and only if the sequence $f(n)$ is $p$-automatic.
\end{theorem}

Kedlaya \cite{K} used generalized power series (see Hahn \cite{H}) to give an extension of Christol's theorem, which has the advantage of giving a complete automaton-theoretic description of the algebraic closure of $\mathbb{F}_q(t)$.  Here we give a brief introduction to the concepts involved. Let $\mathbbm{k}$ be a field.  We define the collection of \emph{generalized Laurent series} over $\mathbbm{k}$, to be the set of elements of the form $\sum_{\alpha\in \mathbb{Q}} f(\alpha) t^{\alpha}$, where $f: \mathbb{Q} \rightarrow \mathbbm{k}$ has the property that $\{\alpha\colon f(\alpha)\neq 0\}$ is a well-ordered subset of $\mathbb{Q}$, where we use the usual order $<$ on $\mathbb{Q}$.  Restricting to maps $f$ with well-ordered support allows us to endow the set of generalized Laurent series over $\mathbbm{k}$ with a ring structure, where addition and multiplication are given respectively by
$$\sum_{\alpha\in \mathbb{Q}} f(\alpha) t^{\alpha} + \sum_{\alpha\in \mathbb{Q}} g(\alpha) t^{\alpha}  =\sum_{\alpha\in \mathbb{Q}} (f+g)(\alpha) t^{\alpha}$$ and
$$\left(\sum_{\alpha\in \mathbb{Q}} f(\alpha) t^{\alpha}\right)\left(\sum_{\alpha\in \mathbb{Q}} g(\alpha) t^{\alpha}\right) = \sum_{\alpha\in \mathbb{Q}}\left( \sum_{\beta\gamma =\alpha} f(\beta) g(\gamma) \right) t^{\alpha}.$$ We refer the reader to Kedlaya \cite{K, K2} for further information.

If the support is contained in $\mathbb{Q}_{\ge 0}$, then we call $f$ a \emph{generalized power series}. We let $\mathbbm{k}((t^{\mathbb{Q}}))$ denote the set of generalized Laurent series over the field $\mathbbm{k}$, and we let $\mathbbm{k}[[t^{\mathbb{Q}}]]$ denote the set of generalized power series over $\mathbbm{k}$. The generalized power series over $\mathbbm{k}$ form a local ring with unique maximal ideal consisting of generalized power series $\sum_{\alpha\ge 0} f(\alpha) t^{\alpha}$ with $f(0)=0$. We let $\mathbbm{k}[[t^{\mathbb{Q}}]]_{>0}$ denote this maximal ideal. We also find it convenient to let 
$\mathbbm{k}((t^{\mathbb{Q}}))_{<0}$ denote the collection of generalized Laurent series over $k$ whose support lies in $(-\infty,0)$.  We now give further details concerning Kedlaya's automaton-theoretic characterization of the algebraic closure of $\mathbb{F}_p(t)$.

Let $k\ge 2$ be a natural number. We say that a string $u=u_1 \dots u_n \in \Sigma^* := \{ 0,1, \dots, k-1, \, _{\bullet} \}^*$, with $_{\bullet}$ representing the radix point, is a \emph{valid} base-$k$ expansion, if $n\ge 1$, $u_1 \neq 0$, $u_n \neq 0$ and exactly one of $u_1, \dots, u_n$ is equal to the radix point. If $u=u_1 \dots u_n$ is a valid base-$k$ expansion and $u_j$ is its radix point, then we can associate a nonnegative $k$-adic rational, $[u]_k$, to $u$ via the rule \begin{equation}\label{eq:valid}
[u]_k = \sum_{i=1}^{j-1} u_i k^{j-1-i} + \sum_{i=j+1}^{n} u_i k^{j-i}.
\end{equation}

We let 
\begin{equation}
S_k := \{ m/k^n \colon m,n \in \mathbb{Z}_{\ge 0}\}.
\end{equation}
 Given a valid base-$k$ expansion $u$ we obtain a value $[u]_k \in S_k$, where we take $[\, _{\bullet} \,]_k=0$, and we say that $u$ is the base-$k$ expansion of $[u]_k$. Conversely, given an element of $S_k$ it has a unique valid base-$k$ expansion and for $v \in S_k$, we write $(v)_k$ for the valid base-$k$ expansion of $v$. Observe that the maps $[\,\cdot \,]_k$ and $(\,\cdot \,)_k$ naturally extend the maps introduced earlier.

A function $f: S_k \rightarrow \Delta$ is $k$-automatic (in Kedlaya's sense) if there is a DFAO $M$ with input alphabet $\Sigma = \{ 0, 1, 2, \dots, k-1, \, _{\bullet} \}$ and output alphabet $\Delta$ such that for each $v \in S_k$, $f(v) = f_M\left( (v)_k \right)$. In analogy with the classical case, a subset of $S_k$ is called a \emph{$k$-automatic set} if its characteristic function is $k$-automatic.

Kedlaya's extension of Christol's theorem uses the notion of quasi-automatic series, which we now define.
\begin{definition} \label{defpquasi}
Let $p$ be a prime and let $q$ be a power of $p$. A generalized Laurent series $\sum_{\alpha\in \mathbb{Q}} f(\alpha)t^{\alpha}\in \mathbb{F}_q((t^{\mathbb{Q}}))$ is \emph{$p$-quasi-automatic} if the following hold: \begin{enumerate}
\item for some integers $a$ and $b$ with $a>0$, the set $aS+b:=\{ai+b \colon i \in S\}$ is contained in $S_p$; and 
\item for some $a,b$ for which (i) holds, the function $f_{a,b}: S_p \rightarrow \mathbb{F}_q$ given by $f_{a,b}(x)=f((x-b)/a)$ is $p$-automatic (in the sense given above).
\end{enumerate}
\end{definition}

Kedlaya's \cite[Theorem 4.1.3]{K} main result is the following extension of Christol's theorem (see also \cite[Theorem 10.4]{K2}).

\begin{theorem}
(Kedlaya) Let $p$ be prime and let $F(t)=\sum_{\alpha\in \mathbb{Q}} f(\alpha)t^{\alpha} \in \bar{\mathbb{F}}_p((t^{\mathbb{Q}}))$ be a generalized Laurent series.  Then if $F(t)$ is algebraic over $\bar{\mathbb{F}}_p(t)$ then there is a power $q$ of $p$ such that $F(t)=\sum_{\alpha\in \mathbb{Q}} f(\alpha)t^{\alpha} \in \mathbb{F}_q((t^{\mathbb{Q}}))$ and $F(t)$ is $p$-quasi-automatic.  Conversely, if $F(t)=\sum_{\alpha\in \mathbb{Q}} f(\alpha)t^{\alpha} \in \mathbb{F}_q((t^{\mathbb{Q}}))$ is $p$-quasi-automatic then $F(t)$ is algebraic over $\bar{\mathbb{F}}_p(t)$.
\label{thmK}
\end{theorem}

\subsection{Algebraic preliminaries}
In this section we provide some of the necessary algebraic background used in our main theorem.


We first introduce places and valuations. Let $\mathbbm{k}$ be a field and let $K$ be a field extension of $\mathbbm{k}$. A \emph{valuation} of $K$ is a map $\nu: K \to \Gamma\cup \{\infty\}$, where $\Gamma$ is a totally ordered abelian group, such that the following conditions are met:
\begin{enumerate}
\item $\nu(a)=\infty$ if and only if $a=0$;
\item $\nu(ab)=\nu(a)+\nu(b)$;
\item $\nu(a+b)\ge \min(\nu(a),\nu(b))$, with equality whenever $\nu(a)\neq \nu(b)$. 
\end{enumerate} 
We say that a valuation is \emph{trivial} if it is zero on all nonzero elements of the field.  We define the \emph{rank} of a valuation to be the rank of the abelian group $\Gamma$; i.e., the dimension of $\Gamma\otimes_{\mathbb{Z}} \mathbb{Q}$ as a $\mathbb{Q}$-vector space, and we say that a valuation is \emph{discrete} if its value group $\Gamma$ is $\mathbb{Z}$. 

Given a valuation, we have a \emph{valuation ring} $\mathcal{O}_{\nu}\subseteq K$ consisting of elements with nonnegative valuation. Then $\mathcal{O}_{\nu}$ is a local ring with a unique maximal ideal, which we denote $\mathcal{M}_{\nu}$, given by the collection of elements with strictly positive valuation. The \emph{residue field} of the valuation is defined to be $\mathcal{O}_{\nu}/\mathcal{M}_{\nu}$.  We say that two valuations of $K$ are \emph{equivalent} if they have the same valuation ring and we call an equivalence class of valuations of $K$ a \emph{place} of $K$.  We will often use an equivalence class representative to represent a place. We will generally deal with discrete valuations $\nu$, in which case the valuation ring $\mathcal{O}_{\nu}$ is a principal ideal domain and a generator for the maximal ideal $\mathcal{M}_{\nu}$ is called a \emph{uniformizing parameter}.


In the case where $\mathbbm{k}$ is an algebraically closed field and $K$ is a finitely generated extension of $\mathbbm{k}$ of transcendence degree one, there is a smooth projective curve $X$ over $\mathbbm{k}$ such that $K$ is the function field of $X$.  Then if we look at non-trivial places of $K$ with trivial restriction to $\mathbbm{k}$, these are parametrized by the closed points of $X$ as follows. Given $x\in X$, we can define a valuation $\nu_x:K\to \mathbb{Z}\cup \{\infty\}$ by taking $\nu_x(f)$ to be the order of vanishing of the function $f$ at the point $x$ (see \cite[Ch. VI, \S17]{ZS}).

In the case we are interested in, $\mathbbm{k}$ will be the algebraic closure of a finite field and in this case, every valuation of $K$ has trivial restriction to $\mathbbm{k}$.
From the above, we then have that the non-trivial places of the field $\bar{\mathbb{F}}_p(t)$ are parametrized by the projective line over $\bar{\mathbb{F}}_p$.

We now use places to define ramification. 
\begin{definition} \label{ur}
Given a finite extension of fields $L\supseteq K$, we say $L$ is \emph{unramified} at a place $\nu$ of $K$ if the value group of every extension of $\nu$ to $L$ is the same as the value group of $\nu$.\end{definition}
In our case, we generally deal with discrete valuations, and a place $\nu$ of $K$ has finitely many extensions $\nu_1,\ldots ,\nu_s$ to $L$. Then for each $i\in \{1,\ldots ,s\}$ we have a discrete valuation ring $\mathcal{O}_{\nu_i}\subseteq L$ consisting of elements in $L$ with nonnegative valuation with $\nu_i$, and similarly we have a discrete valuation ring $\mathcal{O}_{\nu}\subseteq K$.  Then these are local rings whose maximal ideals are principal and we have $\mathcal{O}_{\nu}\subseteq \mathcal{O}_{\nu_i}$.  Then if $\pi$ is a generator for the maximal ideal of $\mathcal{O}_{\nu}$ and if $L$ is unramified at $\nu$ then $\pi$ will also generate the maximal ideal of $\mathcal{O}_{\nu_i}$ for $i=1,\ldots ,s$.  Furthermore, we will often work with places that are parametrized by points in $\mathbb{P}^1$ and so we will often identify places with the corresponding points in projective space.
  
In our setting, the fields $K$ and $K'$ we work with will have the property that the compositum of two finite extensions of $K$ inside $K'$ that are unramified at some place $\nu$ of $K$ is again unramified at $\nu$ (see Serre \cite[Chapter III]{Serre} for further details).

The final algebraic ingredients we will use in proving Theorem \ref{thm:main} are the notion of Artin-Schreier extensions, which are degree-$p$ Galois extensions of fields of positive characteristic $p$, and integrality.  We first quickly recall the relevant definitions for integral elements.

Given integral domains $S\subseteq T$, we say that $u\in T$ is \emph{integral} over $S$ if there is a monic polynomial $f(x)\in S[x]$ with $f(u)=0$. The set of elements of $T$ that are integral over $S$ forms a ring and is called the \emph{integral closure} of $S$ in $T$.  We say that $S$ is \emph{integrally closed} if $S$ is integrally closed in its field of fractions. Finally, we recall the notion of Artin-Schreier extensions. 

\begin{theorem}[Artin-Schreier Theorem] \label{A-S}
Let $p$ be prime, let $\mathbbm{k}$ be a field of characteristic $p$, and let $K$ be a Galois extension of $\mathbbm{k}$ of degree $p$.
\begin{enumerate}
\item[(1)] There exists $\alpha \in K$ such that $K = \mathbbm{k}(\alpha)$ and $\alpha$ is the root of a polynomial $X^p - X - a$ for some $a \in \mathbbm{k}$.
\item[(2)] Conversely, given $a \in \mathbbm{k}$, the polynomial $f(X) = X^p - X -a$ either has one root in $\mathbbm{k}$, in which case all its roots are in $\mathbbm{k}$, or it is irreducible. In this latter case, if $\alpha$ is a root then $\mathbbm{k}(\alpha)$ is cyclic Galois extension of $\mathbbm{k}$ of degree $p$.
\end{enumerate}
\end{theorem}

\begin{proof}
See Lang \cite[VIII.6]{L}.
\end{proof}

The Artin-Schreier theorem provides an inductive means of describing Galois extensions of size a power of $p$.
\begin{remark}
\label{rem:Galois}
Let $\mathbbm{k}$ be a field of characteristic $p>0$ and let $K$ be a Galois extension of $\mathbbm{k}$ of degree $p^m$ for some $m\ge 0$.  Then there exists a chain of fields
$\mathbbm{k}=K_0\subseteq K_1\subseteq \cdots \subseteq K_m=K$ with each $K_i$ a Galois extension of $\mathbbm{k}$ and such that $K_{i+1}$ is an Artin-Schreier extension of $K_i$ for $i=0,\ldots ,m-1$.
\end{remark}
\begin{proof} This follows immediately from the fundamental theory of Galois theory combined with the fact that a group $P$ of order $p^m$ is nilpotent and hence has a chain of subgroups 
$$P=P_0\unrhd P_1\unrhd \cdots \unrhd P_m=\{1\}$$ with each $P_i$ normal in $P$ and $|P_i|=p^{m-i}$ for $i=1,\ldots ,m$.  
\end{proof}
\section{Sparseness}
\label{sec:sparse}
In this section, we give an overview of sparse languages and sparse sets.  Much of the material here is well-known and in some cases we borrow from \cite{BM}.
We let $\Sigma$ be a finite alphabet and we let $\mathcal{L}\subseteq \Sigma^*$ be a language.  We define the counting function of the language
$$f_{\mathcal{L}}(n):=\#\{w\in \mathcal{L}\colon {\rm length}(w)\le n\}.$$
We say that a regular language $\mathcal L$ is {\em sparse} if one of the equivalent conditions in Proposition~\ref{sparse} below hold. Sparse languages play an integral role in the theory of regular languages and finite-state automata and have been studied in numerous contexts. We borrow a summary of conditions equivalent to sparseness from \cite{BM}, which combines results from
\cite{Ginsburg&Spanier:1966,Trofimov:1981,Ibarra&Ravikumar:1986,Szilard&Yu&Zhang&Shallit:1992,Gawrychowski&Krieger&Rampersad&Shallit:2010}.

\begin{proposition}
\label{sparse}
Let $\mathcal{L}$ be a regular language. Then the following conditions are equivalent.
\begin{enumerate}
\item There is a natural number $d$ such that $f_\mathcal{L}(n)={\rm O}(n^d)$.   
\item There is some $C>1$ such that $f_{\mathcal{L}}(n)={\rm o}(C^n)$.
\item There do not exist words $u,v,a,b$ with $a,b$ non-trivial and of the same length and $a\neq b$ such that $u\{a,b\}^* v\subseteq \mathcal{L}$.
\item Suppose $\Gamma = (Q,\Sigma,\delta,q_0,F)$ is an automaton accepting $\mathcal{L}$ in which all states are accessible.
Then $\Gamma$ satisfies the following.
\begin{itemize}
\item[($*$)] If $q$ is a state such that $\delta(q,v)\in F$ for some word $v$ then there is at most one non-trivial word $w$ with the property that $\delta(q,w)=q$ and $\delta(q,w')\neq q$ for every non-trivial proper prefix $w'$ of $w$.
\end{itemize}
\item There exists  an automaton accepting $\mathcal{L}$ that satisfies~{\em($*$)}.
\item The language $\mathcal{L}$ is a finite union of disjoint languages of the form $v_1 w_1^* v_2 w_2^* \cdots v_s w_k^* v_{s+1}$ where $s\ge 0$ and the $v_i$ are possibly trivial words and the $w_i$ are non-trivial words.
\end{enumerate}
\end{proposition}
\begin{proof} See \cite[Proposition 7.1]{BM}.  The disjointness given in item (vi) is not explicitly stated in \cite[Proposition 7.1]{BM}, but it is straightforward to see that the languages one obtains can be taken to be disjoint.
\end{proof}
Given the connection between regular languages over the alphabet $\{0,1,\ldots ,k-1\}$ and $k$-automatic sets given in 
\S\ref{prelim}, we can naturally extend the notion of sparseness to $k$-automatic sets as follows.  Given a subset $S\subseteq \mathbb{N}$, we say that $S$ is a \emph{sparse} $k$-automatic set if $\{(n)_k \in \{0,1,\ldots ,k-1\}^*\colon n\in S\}$ is a sparse sublanguage of $\{0,1,\ldots ,k-1\}^*$.  If one translates conditions (i) and (ii) of Proposition \ref{sparse} into this context, we obtain the following well-known dichotomy.
\begin{theorem} \label{thm:dich}
Let $S \subseteq \mathbb{N}$ be a $k$-automatic set and let $\pi_S(x) = \#\{n\in S \colon n\le x\}$ for $x\ge 0$. Then one of the following alternatives must hold:
\begin{enumerate}
\item[(1)] there exists $d \geq 1$ such that $\pi_S(n) = O((\log n)^d)$ as $n \rightarrow \infty$; or
\item[(2)] there exists a real number $\alpha > 0$ such that $\pi_S(n) > n^{\alpha}$ for all sufficiently large $n$.
\end{enumerate}
\end{theorem}
Then sparse sets are precisely $k$-automatic sets for which there is some $d\ge 1$ such that $\pi_S(n) = O((\log n)^d)$, and this ``gap'' result shows there is a clear delineation between sparse and non-sparse $k$-automatic sets.  We record this more formally.

\begin{definition}
Let $k\ge 2$ be a natural number and let $S$ be a $k$-automatic set. We say that $S$ is \emph{sparse} if condition (1) in Theorem \ref{thm:dich} holds. \label{def:sparse}
\end{definition}

An interesting feature of Proposition \ref{sparse} is that it shows that a sparse language $\mathcal L$ is a finite disjoint union of languages of the form $v_1 w_1^* v_2 w_2^* \dots v_s w_s^* v_{s+1}$ where $s \geq 0$ and the $v_i$ are possibly trivial words and the $w_i$ are non-trivial words. We shall call languages of this special form \emph{simple sparse languages}.  

Then every sparse language is a finite union of disjoint simple sparse languages and if one translates this into sparse $k$-automatic sets, we see that a sparse $k$-automatic set can be written as a disjoint union 
\begin{equation}
\label{eq:sqcup}
S = S_1 \sqcup S_2 \sqcup \cdots \sqcup S_d
\end{equation} for some integer $d \geq 1$, where each $S_i$ is a set of natural numbers of the form 
\begin{equation}
\label{eq: sparse1}
\left\{ [v_1 w_1^{n_1} v_2 w_2^{n_2} \dots v_s w_s^{n_s} v_{s+1}]_k \colon n_1,n_2,\ldots ,n_s\ge 0\right\}.\end{equation} 

We call a set of natural numbers of this form a \emph{simple sparse} $k$-automatic set. Sparse sets are related to \emph{$p$-normal sets} in \cite{Derksen} and that 
simple sparse sets coincide with the sets $$U_p(v_{s+1}, v_{s}, \dots, v_1; w_s,\ldots ,w_1)$$ that are defined in \cite[Definition 7.8]{Derksen}.  

A straightforward computation involving geometric series gives the following remark. 
\begin{remark} \label{rem:sparse}
Let $k\ge 2$ be a natural number and let $S$ be a non-empty simple sparse $k$-automatic set. Then there exist $s\ge 0$, $c_0,\ldots ,c_s\in \mathbb{Q}$ such that $(k^{\ell}-1)c_i\in \mathbb{Z}$ for some $\ell\ge 0$, $c_0 + c_1 + \cdots + c_s \in \mathbb{Z}_{\ge 0}$ and positive integers $\delta_1,\ldots ,\delta_s$ such that
\begin{equation} \label{eq: form}
S = \left\{ c_0 + c_1 k^{\delta_s n_s} + c_2 k^{\delta_s n_s + \delta_{s-1} n_{s-1}}\cdots + c_s k^{\delta_s n_s+\cdots +\delta_1 n_1} \colon n_1,\ldots ,n_s\ge 0\right\}.\end{equation}
Moreover $n\ge c_0$ for all $n\in S$ and $c_0\in S$ if and only if $s=0$.
\end{remark}

\begin{proof} Let $[v_i]_k=a_i$ for $i=1, \ldots, s+1$ and $[w_i]_k=b_i$ for $i=1, \ldots, s$. Note that each $b_i$ is strictly positive  since the $w_i$ are non-trivial words. Let $\mu_i = \length(v_i) \ge 0$ and $\delta_i = \length(w_i) \geq 1$. Then we compute the value of $[v_1 w_1^{n_1} v_2 w_2^{n_2} \cdots v_s w_s^{n_s} v_{s+1}]_k $ as follows.
\begin{align*}
& [v_1 w_1^{n_1} v_2 w_2^{n_2} \cdots v_s w_s^{n_s} v_{s+1}]_k\\
&= a_{s+1} + k^{\mu_{s+1}} \left( b_s + k^{\delta_{s}} b_s + \cdots + k^{\delta_{s} (n_s-1)} b_s \right) +k^{\mu_{s+1} + n_s \delta_s} a_s \\
&+ k^{\mu_{s+1} + \mu_s + n_s \delta_s} \left( b_{s-1} + k^{\delta_{s-1}} b_{s-1} + \cdots + k^{\delta_{s-1} (n_{s-1}-1)} b_{s-1} \right) + \cdots\\
&+ k^{\mu_{s+1} + \mu_{s} + \mu_{s-1} + \cdots + \mu_2 + n_s \delta_s + \cdots + n_2 \delta_2} \left( b_1 + k^{\delta_1} b_1 + \cdots + k^{\delta_1 (n_1 - 1)} b_1 \right) \\
&+ a_1 k^{\mu_{s+1} + \cdots + \mu_2 + \delta_s n_s + \cdots + \delta_1 n_1} \\
&= a_{s+1} + k^{\mu_{s+1}} b_s \left( \frac{k^{n_s \delta_s}-1}{k^{\delta_s}-1} \right)+ k^{\mu_{s+1} + n_s \delta_s} a_s + k^{\mu_{s+1} + \mu_s + n_s \delta_s} b_{s-1} \left( \frac{k^{\delta_{s-1} n_{s-1}}-1}{k^{\delta_{s-1}}-1} \right) + \cdots \\
&+ k^{\mu_{s+1} + \mu_s + \cdots \mu_2 + n_s \delta_s + \cdots + n_2 \delta_2} b_1 \left( \frac{k^{\delta_1 n_1}-1}{k^{\delta_1}-1} \right) + a_1 k^{\mu_{s+1} + \mu_s + \cdots + \mu_2 + n_s \delta_s + \cdots + n_1 \delta_1}.
\end{align*}
The result follows, taking $c_0 := a_{s+1} - \frac{k^{\mu_{s+1}}b_s}{k^{\delta_s}-1}$, 
$$c_1 := \frac{k^{\mu_{s+1}}b_s}{k^{\delta_s}-1}+k^{\mu_{s+1}}a_s - \frac{k^{\mu_{s+1}+\mu_s} b_{s-1}}{k^{\delta_{s-1}}-1}, \dots, c_s := a_1 k^{\mu_{s+1} + \mu_s + \cdots + \mu_2} + \frac{k^{\mu_{s+1} + \mu_s + \cdots + \mu_2}b_1}{k^{\delta_1}-1}.$$
We observe that if $s\ge 1$ then $c_0<a_{s+1}=[v_{s+1}]_k$ and since every element of $S$ is at least as large as $[v_{s+1}]_k$ we then see that $n\ge c_0$ for every $n\in S$ and $c_0\in S$ if and only if $s=0$.  Taking $n_1=n_2=\cdots =n_s=0$ and using the fact that $S$ consists of nonnegative integers, we see that $c_0+\cdots +c_s\in \mathbb{Z}_{\ge 0}$.
\end{proof}

\begin{definition}
Let $p$ be a prime and let $q$ be a power of $p$.  Given an algebraic power series $F(t) = \sum_{n \ge 0} f(n)t^n \in \mathbb{F}_q[[t]]$, we call $F(t)$ \emph{sparse} if the support of $F(t)$ is a sparse $p$-automatic set; that is, if $\{n \colon f(n) \neq 0 \}$ is sparse.
\end{definition}



We now extend the notion of sparseness to subsets of $S_k\subseteq \mathbb{Q}$ with $k\ge 2$ a natural number. Following Kedlaya \cite{K}, we work with the alphabet $\{0,1,\ldots ,k-1,\, _{\bullet}\}$, where $_{\bullet}$ represents the radix point in the base-$k$ expansion of a $k$-adic rational.  The set of valid base-$k$ expansions is a regular language \cite[Lemma 2.3.2]{K}, where such expansions are given by the language 
\begin{equation}
\label{eq:E}
\mathcal{E}_k := \{ u = u_1 u_2 \dots u_n \in \Sigma^* \colon n\ge 1, u_1 \neq 0, u_n \neq 0, \text{ exactly one of } u_1, \dots, u_n \text{ is equal to}~_{\bullet} \}.
\end{equation}

By the definition of sparseness for languages, a sublanguage $\mathcal L$ of $\mathcal{E}_k $ is sparse if $f_{\mathcal{L}}(n) = {\rm O}(n^d)$ for some $d \ge 1$. If $\mathcal L$ is sparse then by Proposition \ref{sparse}, it is a finite union of languages of the form $u_1 w_1^* u_2 w_2^* \dots w_s^* u_{s+1}$, where the $u_i$ are possibly trivial and the $w_i$ are non-trivial words. Furthermore by the definition of the language $\mathcal{E}_k $, exactly one of $\{u_1, \dots, u_{s+1}\}$ contains the radix point and none of the $\{ w_1, w_2, \dots, w_s \}$ can contain the radix point. Hence for a language $u_1 w_1^* u_2 w_2^* \dots w_s^* u_{s+1}$, there is a unique index $j$, $1 \le j \le s+1$, such that $u_j$ contains the radix point and so we can write this $u_j$ as $u_j'\, _{\bullet} u_j''$, and a sparse sublanguage $\mathcal{L}$ of $\mathcal{E}_k $ can be expressed as a finite disjoint union of languages of the form $$u_1 w_1^* u_2 w_2^* \dots w_{j-1}^* u_j'\,  _{\bullet}u_j'' w_j^* u_{j+1} \dots w_s^* u_{s+1}.$$ For our applications, we will be interested in the case where $k=p$ is prime and we are viewing elements of $\mathcal{E}_p$ as base-$p$ expansions of elements of $S_p$ via the map
$[\, \cdot \,]_{p}$ given in Equation (\ref{eq:valid}). In analogy with our definition of sparse subsets of the natural numbers, we will say that a subset of $S_p$ of the form
\[ \left\{[u_1 w_1^{n_1} u_2 w_2^{n_2} \dots w_{j-1}^{n_{j-1}} u_j' \, _{\bullet} u_j'' w_j^{n_j} u_{j+1} \dots w_s^{n_s} u_{s+1}]_p \colon n_1, n_2, \dots, n_s \ge 0\right\} \]
is a \emph{simple sparse} subset of $S_p$, and we will say that a subset $S\subseteq S_p$ is a \emph{sparse} subset of $S_p$ if $S$ is a finite union of simple sparse subsets of $S_p$. If $S\subseteq \mathbb{N}$ then it is immediate that being sparse as a subset of $S_p$ exactly coincides with the notion of sparseness introduced for $p$-automatic subsets of the natural numbers given in Definition \ref{def:sparse}.  

We now shift our focus back to generalized Laurent series. Given a $p$-quasi-automatic generalized Laurent series $F(t)=\sum_{\alpha\in \mathbb{Q}} f(\alpha) t^{\alpha}\in \bar{\mathbb{F}}_p((t^{\mathbb{Q}}))$ with support $S\subseteq \mathbb{Q}$, there is a power $q$ of $p$ such $f(\alpha)\in \mathbb{F}_q$ by Theorem \ref{thmK} and there are integers $a>0$ and $b$ such that $aS+b \subseteq S_p$ and $f_{a,b}: S_p \rightarrow \mathbb{F}_q$ given by $f_{a,b}(x) = f((x-b)/a)$ is $p$-automatic in the sense of Definition \ref{defpquasi}. We will say that $F(t)$ is a  \emph{sparse generalized Laurent series} if $aS+b \subseteq S_p$ is a sparse subset of $S_p$, where $a$ and $b$ are as above. 
\begin{remark}\label{rem:count}
Let $S\subseteq S_p$ be $p$-automatic.  Then $S$ is sparse if and only if $$\#\{a\in S \colon a<p^n ~{\rm and}~p^n a\in \mathbb{N}\} = {\rm O}(n^d)$$ for some positive integer $d$.
\end{remark}
\begin{proof} Let $\mathcal{L}\subseteq \mathcal{E}_k$ be the regular language $\{(x)_p \colon x\in S\}$.  Then $S$ is sparse if and only if $\mathcal{L}$ is a sparse.  Notice that 
$$\#\{a\in S \colon a< p^n ~{\rm and}~p^n a\in \mathbb{N}\}=\#\{u\, _{\bullet} v \in \mathcal{L} \colon {\rm length}(u),{\rm length}(v)\le n\}.$$
The set $\{u\, _{\bullet} v \in \mathcal{L} \colon {\rm length}(u),{\rm length}(v)\le n\}$ is a subset of the set of words in $\mathcal{L}$ of length at most $2n+1$ and so if $\mathcal{L}$ is sparse, $\#\{a\in S \colon a<p^n ~{\rm and}~p^n a\in \mathbb{N}\} = {\rm O}(n^d)$ for some $d\ge 0$.  Conversely, if $\mathcal{L}$ is not sparse, then it contains a sublanguage of the form
$u\{y,z\}^* v$, where exactly one of $u$ and $v$ contains the radix point.  Let $\kappa$ denote the maximum of the lengths of $y$ and $z$.  Then in either case, that every element of the form $[uwv]_p$, with $w$ a word in $\{y,z\}^*$ of length at most $(2\kappa)^{-1} n$, is in
$\{a\in S \colon a<p^n ~{\rm and}~p^n a\in \mathbb{N}\}$ for $n$ sufficiently large.  Since the number of words of length at most $(2\kappa)^{-1} n$ in $\{y,z\}^*$ grows exponentially in $n$,  $\#\{a\in S \colon a<p^n ~{\rm and}~p^n a\in \mathbb{N}\} \neq {\rm O}(n^d)$ when $\mathcal{L}$ is not sparse.
\end{proof}
\begin{definition}
Let $S$ be a (not necessarily $p$-automatic) subset of $S_p$. We say that $S$ is \emph{weakly sparse} if $$\#\{a\in S \colon a<p^n ~{\rm and}~p^n a\in \mathbb{N}\} = {\rm O}(n^d)$$ for some positive integer $d$. In particular, a subset of $S_p$ is sparse if and only if it is $p$-automatic and weakly sparse, and an automatic subset of a weakly sparse set is sparse.
\end{definition}

The following remark follows immediately from Remark \ref{rem:count}
\begin{remark} \label{affine}
If $a,a'$ and $b,b'$ are rational numbers with $a,a'>0$ then if $S\subset \mathbb{Q}$ has the property that both $aS+b$ and $a'S+b'$ lie in $S_p$ then $aS+b$ is sparse if and only if $a'S+b'$ is sparse, and so this definition of sparseness does not depend upon the choice of affine transformation that we use to push $S$ into the $p$-adic rationals.  
\end{remark}

  Before the next remark we recall that $\mathbb{Z}_{(p)}$ is the subring of rational numbers of the form $a/b$ with $a,b$ integers and $p\nmid b$.

\begin{remark}
Let $p$ be a prime number and let $S$ be a non-empty well-ordered simple sparse subset of $S_p$. Then there exist $s\ge 0$ and $c_0,\ldots ,c_{j-1}\in \mathbb{Z}_{(p)}$ and $d_{j-1},\ldots ,d_s\in \mathbb{Q}$ and positive integers $\delta_1,\ldots ,\delta_s$ such that
\begin{equation}
\label{eq:split}
\begin{split}
S = \{c_0 &+ c_1 p^{\delta_{j-1} n_{j-1}} + c_2 p^{\delta_{j-1} n_{j-1} + \delta_{j-2} n_{j-2}}\cdots + c_{j-1} p^{\delta_{j-1} n_{j-1}+\cdots +\delta_1 n_1} \\
&+ d_{j-1} + d_j p^{-\delta_j n_j} + d_{j+1} p^{-(\delta_j n_j + \delta_{j+1} n_{j+1})} + \cdots+ d_{s} p^{-(\delta_j n_j + \cdots + \delta_s n_s)} \colon n_1,\ldots ,n_s\ge 0 \}.
\end{split}
\end{equation}
Furthermore, we have
$$c_1 p^{\delta_{j-1} n_{j-1}} + c_2 p^{\delta_{j-1} n_{j-1} + \delta_{j-2} n_{j-2}}\cdots + c_{j-1} p^{\delta_{j-1} n_{j-1}+\cdots +\delta_1 n_1} \ge 0$$ for all $n_1,\ldots ,n_{j-1}\ge 0$ and
$$ d_j p^{-\delta_j n_j} + d_{j+1} p^{-(\delta_j n_j + \delta_{j+1} n_{j+1})} + \cdots+ d_{s} p^{-(\delta_j n_j + \cdots + \delta_s n_s)}\le 0$$ for all $n_j,\ldots ,n_s\ge 0$. 
\label{rem39}
\end{remark}

\begin{proof} Let $[u_i]_p=a_i$ and $\mu_i=\length(u_i)$ for $i \in \{1, \dots, s\} \setminus \{j\}$, $[u_j']_p=a_j'$, $[u_{j}'']_p=a_{j}''$, $\mu_j'=\length(u_j')$, $\mu_j''=\length(u_j'')$, and $[w_i]_p=b_i$ and $\delta_i = \length(w_i)$ for $i=1, \dots, s$. The pre-radix part can be handled as in Remark \ref{rem:sparse} and the post-radix part is handled similarly as follows. Putting
$d_{j-1} = a_j'' - \frac{b_j p^{-\mu_j''}}{p^{-\delta_j}-1}$, $d_j=\frac{p^{-\mu_j''} b_j}{p^{-\delta_j}-1}+p^{-\mu_j''}a_{j+1}-\frac{b_{j+1}p^{-\mu_j''-\mu_{j+1}}}{p^{-\delta_{j+1}}-1}$, \dots, 
$d_s = p^{-\mu_j''-\mu_{j+1} - \cdots - \mu_{s}} a_{s+1} + \frac{b_s p^{-\mu_j'' - \mu_{j+1} - \cdots - \mu_s}}{p^{-\delta_s}-1}$, we get the desired description of $S$.  The inequalities $$c_1 p^{\delta_{j-1} n_{j-1}} + c_2 p^{\delta_{j-1} n_{j-1} + \delta_{j-2} n_{j-2}}\cdots + c_{j-1} p^{\delta_{j-1} n_{j-1}+\cdots +\delta_1 n_1} \ge 0$$ for all $n_1,\ldots ,n_{j-1}\ge 0$ and
$$ d_j p^{-\delta_j n_j} + d_{j+1} p^{-(\delta_j n_j + \delta_{j+1} n_{j+1})} + \cdots+ d_{s} p^{-(\delta_j n_j + \cdots + \delta_s n_s)}\le 0$$ for all  for all $n_j,\ldots ,n_{s}\ge 0$ come from the fact that $S$ is well-ordered.  To obtain the first inequality, suppose that $$\Psi(n_1,\ldots ,n_{j-1}):=c_1 p^{\delta_{j-1} n_{j-1}} + c_2 p^{\delta_{j-1} n_{j-1} + \delta_{j-2} n_{j-2}}\cdots + c_{j-1} p^{\delta_{j-1} n_{j-1}+\cdots +\delta_1 n_1} <0$$ for some $n_1,\ldots ,n_{j-1}\ge 0$.  Then since
$$\Psi(n_1,\ldots ,n_{j-1}+a)=p^{\delta_{j-1} a} \Psi(n_1,\ldots ,n_{j-1})$$ for $a\ge 0$ and $\delta_{j-1} >0$, we obtain an infinite descending subsequence in $S$, contradicting the fact that it is well-ordered. The second inequality follows in a similar manner.
\end{proof}

The collection of sparse series forms a subalgebra of the ring of algebraic power series with coefficients in $\bar{\mathbb{F}}_p$.  This is in fact rather straightforward, but for the sake of completeness, we include a proof (see Proposition \ref{propsparse}); in addition, we show that sparse series possess natural closure properties, which we detail below. 
\begin{definition} \label{closureprops}
Let $B\subseteq C$ be subalgebras of the ring of generalized Laurent series $\bar{\mathbb{F}}_p((t^{\mathbb{Q}}))$. We say that $B$ is \emph{Artin-Schreier closed in} $C$ if the following hold:
\begin{enumerate}
\item[(P1)] if $F(t)\in B$ and if $G(t)\in C$ is a solution to the equation $X^p-X+F(t)=0$, then $G(t) \in B$;
\item[(P2)] If $F(t)\in B$ and $\alpha\in \bar{\mathbb{F}}_p$ then $F(\alpha t)\in B$;
\item[(P3)] if $F(t) \in B$, $c\in \mathbb{Q}_{>0}$ and $d\in \mathbb{Q}$, and $t^dF(t^c) \in C$, then $t^dF(t^c)\in B$.
\end{enumerate}
\label{def:AS}
\end{definition}
We make a remark concerning property (P1). In general, a generalized Laurent series $F(t)$ can be written as $F_{+}(t)+c + F_{-}(t)$, where $c$ is constant, $F_{+}\in \bar{\mathbb{F}}_p[[t^{\mathbb{Q}}]]_{>0}$ and $F_{-}(t) \in \bar{\mathbb{F}}_p((t^{\mathbb{Q}}))_{<0}$.  Then there is some $a\in \bar{\mathbb{F}}_p$ such that $a^p-a=c$ and all solutions to $X^p-X=-F(t)$ are of the form $G_{+}(t)+G_{-}(t)+a+i$, where $i\in \mathbb{F}_p$, $G_{+}(t)= F_{+}(t)+F_{+}(t^p)+F_{+}(t^{p^2})+\cdots $ and $G_{-}(t)=-F_{-}(t^{1/p})-F_{-}(t^{1/p^2}) -\cdots$.  
In the case when $C$ is the ring of formal power series and $B$ is a subalgebra of $C$, condition (P1) simply says that if $F(t)\in B$ and $F(0)=0$ then $F(t)+F(t^p)+F(t^{p^2})+\cdots $ is also in $B$.  Our goal is to show that various rings of sparse algebraic series are Artin-Schreier closed in natural overrings.  To do this, we need a quick lemma about sparse subsets of $S_p$.
\begin{lemma} Let $p$ be prime, let $b$ be a positive integer and let $S\subseteq S_p$ be a well-ordered sparse set. Then we have the following:
\label{LEM1}
\begin{enumerate}
\item[(a)] $S\cap [0,b)$ and $S\cap (b,\infty)$ are both sparse;
\item[(b)] if $T:=S\cap (b,\infty)$ then $\bigcup_{n\ge 0} \left((T-b)p^n+b\right)$ is a well-ordered sparse set;
\item[(c)] if $U:=S\cap [0,b)$ then $\bigcup_{n\ge 1} \left((U-b)p^{-n}+b\right)$ is a well-ordered sparse set.
\end{enumerate}
\end{lemma}
\begin{proof} To prove (a), let $S \subseteq S_{p}$ be a sparse set. Applying the map $(\,\cdot\,)_p$ to $S$ we obtain a sparse regular language $\mathcal{L} \subseteq \mathcal{E}_{p}$. Both $S_p \cap [0,b)$ and $S_p \cap (b, \infty)$ correspond to regular languages and hence $S \cap [0,b)$ and $S \cap (b, \infty)$ correspond to regular languages as well. Moreover, since they are sublanguages of the sparse language $\mathcal{L}$, both $S\cap [0,b)$ and $S\cap (b,\infty)$ are sparse.

We now prove part (b). By part (a), we know that $T$ is sparse. By Remark \ref{affine}, affine transformations preserve sparseness, so $T':=T-b$ is a well-ordered sparse set. Since $T'$ is a finite disjoint union of simple sparse sets, it is no loss of generality to assume that $T'$ is a simple sparse set in what follows.
We let $\mathcal{L}_{T'}$ denote the sublanguage of $\mathcal{E}_p$ obtained by applying $(\,\cdot \,)_p$ to $T'$.
Then since we are assuming that $T'$ is simple sparse, $\mathcal{L}_{T'}$ is a language of the form
$$u_1 w_1^* u_2 w_2^* \dots w_{j-1}^* u'_j\,  _{\bullet}u''_j w_j^* u_{j+1} \dots w_s^* u_{s+1}.$$
Consequently, the sublanguage of $\mathcal{E}_p$ obtained by applying $(\,\cdot \,)_p$ to $\bigcup p^n T'$ is a finite union of languages of the form \begin{enumerate}
\item $u_1 w_1^* u_2 w_2^* \cdots u_s w_s^* u_{s+1}0^* \, _{\bullet} \,$; 
\item $u_1 w_1^* u_2 w_2^* \dots w_{j_0-1}^* u'_{j_0}\,  _{\bullet}u''_{j_0} w_{j_0}^* u_{j_0+1} \dots w_s^* u_{s+1}$; and
\item $u_1 w_1^* u_2 w_2^* \dots w_{i_0-1}^* u_{i_0} w_{i_0}^* w'_{i_0}\,  _{\bullet} w''_{i_0} w_{i_0}^* u_{i_0+1} \dots w_s^* u_{s+1}$,
\end{enumerate}
where $i_0\in \{j,\ldots ,s\}$ and $w'_{i_0}w''_{i_0} = w_{i_0}$, and $j_0\in \{j,\ldots, s+1\}$, $u'_{j_0}u''_{j_0} = u_{j_0}$, and if $j_0=j$ then $u_{j_0}'$ has $u_j'$ as a prefix.
Hence $\bigcup p^n T'$ is a sparse set and since affine transformations preserve sparseness, we have $$b+\bigcup p^n T'=\bigcup_{n\ge 0} \left((T-b)p^n+b\right)$$ is sparse.
To see that $\bigcup_{n\ge 0} \left((T-b)p^n+b\right)$ is well-ordered, it suffices to show that the union of $p^n T'$ is well-ordered. Let $t_0>0$ denote the smallest element of $T'$ and suppose that $x_1\ge x_2\ge \cdots $ is a weakly decreasing chain in the union of the sets $p^n T'$.  Then there is some $N>0$ such that $x_1< p^N t_0$ and hence $x_1,x_2,x_3,\ldots$ must be contained in the finite union $\bigcup_{i< N} p^i T'$, which is well-ordered, as it is a finite union of well-ordered subsets of $\mathbb{Q}$, and so the chain $x_1\ge x_2\ge \cdots$ necessarily terminates. Thus we have established part (b).

Finally, we prove part (c). The proof that $U':=\bigcup_{n\ge 1} \left((U-b)p^{-n}+b\right)$ is well-ordered is done exactly as in the proof of part (b).  Thus it only remains to show that this set is sparse.  We write $U=\bigcup_{a=0}^{b-1} U_a$, where $U_a=\{x\in U\colon a\le x<a+1\}$.
Then every element in $U_a$ has a base-$p$ expansion of the form $(a)_p\, _{\bullet}\, w$, where $w$ lies in a sparse sublanguage $\mathcal{C}_a$ of $\{0,1,\ldots ,p-1\}^*$. 
Then $b-U_a = (b-(a+1)) + \{ 1- [_{\bullet} w]_p  \colon w\in \mathcal{C}_a\}$. Since $\mathcal{C}_a$ is sparse, it is a finite union of simple sparse languages.  Observe, moreover, that if $x\in [0,1)\cap S_p$ is a number having base-$p$ expansion $_{\bullet} \, v_1 w_1^{*} v_2 w_2^{*} \cdots v_s w_s^{*} v_{s+1}$, with $v_{s+1}$ non-empty, then $1-x$ has base-$p$ expansion 
$$ _{\bullet} \bar{v}_2 \bar{w}_1^{*} \bar{v}_2 \bar{w}_2^{*} \cdots \bar{v}_s \bar{w}_s^{*} \tilde{v}_{s+1},$$
where if $u = a_1 a_2 \dots a_{d}$, $d \ge 1$, we define $\bar{u} := (p-1-a_1)(p-1-a_2) \dots (p-1-a_d)$ and $\tilde{u} := (p-1-a_1)(p-1-a_2) \dots (p-1-a_{d-1})(p-a_d)$.  In general, if $v_{s+1}$ is empty, one can get a similar description of the set of $1-x$ and, in this way, one can show that $\{ 1- [_{\bullet} w]_p  \colon w\in \mathcal{C}_a\}$ is sparse (although it need not be well-ordered).  Thus $b-U=\bigcup_{a=0}^{b-1} (b-U_a)$ is sparse.
Now we write $$b-U = \bigcup_{0 \le a < b} \left(a + [\, _{\bullet} \, \mathcal{D}_{a}]_p\right),$$ where each $\mathcal{D}_a$ is a sparse sublanguage of $\{0,1,\ldots ,p-1\}^*$.
If the base $p$-expansion of $a$ is equal to $a_1\cdots a_r$, then 
$$\bigcup_{n\ge1}p^{-n}  \left(a + [\, _{\bullet} \, \mathcal{D}_{a}]_p\right) = \bigcup_{i=0}^{\infty} [\, _{\bullet} 0^i a_1\cdots a_r\,_{\bullet}\, \mathcal{D}_a]_p \cup  \bigcup_{i=1}^{ r} [a_1 a_2 \dots a_i \, _{\bullet} \, a_{i+1} \dots a_r \mathcal{D}_a]_p,$$ which is a finite union of simple sparse sets, because $\mathcal{D}_a$ is a finite union of simple sparse languages.
Hence $\bigcup_{n\ge0} (b-U)p^{-n}$ is a finite union of sparse sets and thus is itself sparse. Now 
$$U'=\bigcup_{n\ge 1} \left((U-b)p^{-n}+b\right) = b - \bigcup_{n\ge1} (b-U)p^{-n} \subseteq [0,b),$$
which is sparse by the same argument as above. The result follows.
\end{proof}

We need one more basic fact.  
\begin{lemma} Let $S, T\subseteq S_p$ be well-ordered sparse sets. Then $S\cup T$ and $S+T$ are well-ordered sparse.
\label{lem:sum}
\end{lemma}
\begin{proof}[Proof of Lemma \ref{lem:sum}]
By work of Kedlaya \cite[Lemmas 7.2.1 and 7.2.2]{K}, we have that $S\cup T$ and $S+T$ are $p$-automatic; moreover, they are well-ordered \cite[Lemma 3.1.4]{K}.  Since sparse sets can be decomposed as a finite union of simple sparse sets, and since a finite union of simple sparse sets is sparse, we see $S\cup T$ is sparse.  We now show that $S+T$ is weakly sparse, from which it will immediately follow that $S+T$ is sparse.  Since a finite union of weakly sparse sets is weakly sparse and since $S$ and $T$ are finite unions of simple sparse sets, we may assume without loss of generality that $S$ and $T$ are simple sparse.

A simple sparse subset of $S_p$ is of the form $\{[u\, _{\bullet} v]_p \colon u\in \mathcal{L}, v\in \mathcal{L}'\}$ where $\mathcal{L},\mathcal{L}'$ are simple sparse sublanguages of $\{0,1,\ldots ,p-1\}^*$.  In particular, $S= A+X$, where $A$ is a sparse subset of $\mathbb{N}$ and $X$ is a well-ordered sparse subset of $S_p\cap [0,1)$. Similarly, $T= B+Y$, where $B$ is a sparse subset of $\mathbb{N}$ and $Y$ is a sparse well-ordered subset of $S_p\cap [0,1)$. 
Hence $S+T = (A+B)+(X+Y)$.  Since $A$ and $B$ are sparse sets of natural numbers, and since $\pi_{A+B}(x) \le \pi_{A}(x)\pi_{B}(x)$, we see that $\pi_{A+B}(x) ={\rm O}((\log \, x)^d)$ for some $d\ge 0$.  In particular, since a sum of two $p$-automatic subsets of natural numbers is again a $p$-automatic set of natural numbers \cite[Theorem 5.6.3]{AS}, $A+B$ is a sparse $p$-automatic subset of $\mathbb{N}$. Now let $\mathcal{L}\subseteq \mathcal{E}_p$ be the regular language $\{(x)_p \colon x\in S+T\}$.  Since $X+Y\subseteq [0,2)$ and is well-ordered \cite[Lemma 3.1.4]{K}, if $u\, _{\bullet} v\in \mathcal{L}$ then $[u]_p\in (A+B)\cup (A+B+1)$ for $u\, _{\bullet} v\in \mathcal{L}$; moreover, by a result of Kedlaya \cite[Theorem 7.1.6]{K}, there is a sparse $p$-automatic subset $Z$ of $S_p\cap [0,1)$ such that $\{[v]_p \colon u\, _{\bullet} v\in \mathcal{L}\}=Z$.  Thus 
$$\# \{ a\in S+T\colon a<p^n~{\rm and}~p^n a\in \mathbb{N}\} \le \pi_{(A+B)\cup (A+B+1)}(p^n)\cdot \#\{a\in Z\colon p^n a\in \mathbb{N}\}.$$ Then by the above remarks and Remark \ref{rem:count}, both $\pi_{(A+B)\cup (A+B+1)}(p^n)$ and $\#\{a\in Z\colon p^n a\in \mathbb{N}\}$ are bounded by polynomials in $n$ and thus we obtain the result from Remark \ref{rem:count}.
\end{proof}

\begin{proposition} \label{propsparse}
Let $p$ be prime.  Then we have the following:
	\begin{enumerate}
		\item[(a)] the collection of sparse algebraic power series in $\bar{\mathbb{F}}_p[[t]]$ forms a subalgebra of the ring of algebraic power series; moreover this subalgebra is Artin-Schreier closed in $\bar{\mathbb{F}}_p[[t]]$;
		\item[(b)] the collection of sparse algebraic generalized series in $\bar{\mathbb{F}}_p((t^{\mathbb{Q}}))$ forms a subalgebra of the ring of generalized Laurent series; moreover this subalgebra is Artin-Schreier closed inside $\bar{\mathbb{F}}_p((t^{\mathbb{Q}}))$.
	\end{enumerate}
\end{proposition}

\begin{proof}
We let $A$ denote the collection of sparse algebraic power series in $\bar{\mathbb{F}}_p[[t]]$.  Then to show that $A$ is a subalgebra, it is sufficient to show that it is closed under sum and multiplication.
Let $F(t),G(t) \in A$ and let $S_F$ and $S_G$ denote the supports of $F$ and $G$ respectively. Then the support of $(F+G)(t)$ is contained in $S_F\cup S_G$ and since $S_F$ and $S_G$ are sparse then we see from the characterization of sparseness given in Theorem \ref{thm:dich} that $S_F\cup S_G$ is sparse and so $F+G$ is a sparse algebraic power series.
The support of $F(t)G(t)$ is contained in $S_F+S_G$, where $S_F+S_G$ is the collection of natural numbers that can be expressed in the form $a+b$ with $a\in S_F$ and $b\in S_G$.
If we define $\pi_S(x) = \#\{n\le x\colon n\in S\}$ for $x\ge 0$ for a subset $S$ of the natural numbers, then
$\pi_{S_F+S_G}(x) \le \pi_{S_F}(x)\pi_{S_G}(x)$ and so again from the characterization of sparseness given in Theorem \ref{thm:dich}, $F(t)G(t)$ is a sparse algebraic power series. 
The only property from $(P1)$--$(P3)$ in Definition \ref{def:AS} that does not obviously hold for sparse series is property (P1).  Suppose that $F(t)$ is a sparse series with $F(0)=0$ and let $S$ be the support of $F$. Then if we let $S'$ denote the 
support of $G(t):=F(t)+F(t^p)+\cdots $, then we see that $S'$ is contained in $T:=\cup_{n\ge 0} (p^n\cdot S)$. 
 Then if we look at the language 
$\mathcal{L}:=\{(n)_k \colon n\in S\}$ then $\{(n)_k \colon n\in S'\}\subseteq \mathcal{L}\cdot \{0\}^*$, which is sparse by by Equations (\ref{eq:sqcup}) and (\ref{eq: sparse1}).  Since sparse languages are closed under the process of taking regular sublanguages, the support of $G(t)$ is sparse.

For part (b), we must show that if $F(t), G(t)\in \bar{\mathbb{F}}_p((t^{\mathbb{Q}}))$ are sparse then so are $F(t)+G(t)$ and $F(t) G(t)$.  Let $S_F$ and $S_G$ denote the supports of $F$ and $G$ respectively. After replacing $F$ and $G$ by $t^b F(t^a)$ and $t^b G(t^a)$ for some positive rational numbers $a$ and $b$, we may assume that $S_F, S_G\subseteq S_p$.  Then the supports of $F(t)+G(t)$ and $F(t)G(t)$ are contained in $S_F\cup S_G$ and $S_F+S_G$ respectively, and since the supports of $F(t)+G(t)$ and $F(t)G(t)$ are $p$-automatic and well-ordered, we then see they are sparse by Lemma \ref{lem:sum}.

To show the property of being Artin-Schreier closed holds, it is again enough to prove that (P1) holds. Let $F(t)$ be a sparse generalized power series and again let $S$ denote its support. By assumption there are integers $a$ and $b$ with $a,b>0$ such that $T:=Sa+b\subseteq S_p$ is automatic, sparse, and well-ordered. By Lemma \ref{LEM1}, $T_{+}:=T\cap (b,\infty)$ and $T_{-}:=T\cap [0,b)$ are both sparse automatic subsets of $S_p$.  Let $S_{+}= (T_{+}-b)/a$ and $S_{-}=(T_{-}-b)/a$.  By the remarks following Definition \ref{closureprops}, if $G(t)$ is a solution to the equation $X^p-X+F(t)=0$ and if $U$ denotes the support of $G$, then $U$ is contained in the union of $S_1:=\bigcup_{n\ge 0} p^n S_{+}$, $S_2:=\bigcup_{n\ge 1} p^{-n} S_{-}$, and $\{0\}$. Then let $T_1=S_1 a +b\subseteq S_p$ and let $T_2=S_2a+b$. Then since automatic subsets of sparse sets are sparse, it suffices to show that $T_1$ and $T_2$ are sparse.  But $T_1=\bigcup_{n\ge 0} ((T_{+}-b)p^n+b)$ and $T_2=\bigcup_{n\ge 1} \left((T_{-}-b)p^{-n}+b\right)$, so $T_1$ and $T_2$ are sparse by Lemma \ref{LEM1}, and so the result follows.
\end{proof}


\section{Proof of Theorem \ref{thm:main} $({\it a})$} \label{mainres}
In this section and the next, we will complete the proof of Theorem \ref{thm:main}.  In fact, we will prove a somewhat more general version of Theorem \ref{thm:main} that deals with Kedlaya's extension of Christol's theorem.  Before giving this more general statement, we find it convenient to fix the following notation.  

\begin{notation} \label{not} We adopt the following notation:
\begin{enumerate}[(1)]
\item we let $p$ be prime and we let $q$ be a power of $p$;
\item we let $K=\bar{\mathbb{F}}_p(t^{1/n}\colon n\ge 1, p\nmid n)$;
\item we let $R$ denote $\mathbb{F}_p[t^{\pm 1/n} \colon n\ge 1, p\nmid n]$;
\item we let $L$ denote the compositum of all Galois extensions of $K$ of order a power of $p$;
\item we let $L_0$ denote the elements $G\in L$ such that $K[G(t)]/K$ is unramified outside $0$ and $\infty$;
\item for each $\lambda\in \mathbb{P}^1_{\bar{\mathbb{F}}_p}$, we let $\nu_{\lambda}$ be the valuation of $K$ induced by taking the order of vanishing at $t=\lambda$ (this valuation is discrete when $\lambda\in \bar{\mathbb{F}}_p^*=\mathbb{P}^1\setminus \{0,\infty\}$);
\item given a finite Galois extension $E$ of $K$, we let $\mathcal{V}_{E}\subseteq E$ be the set of elements $a\in E$ such that $\nu(a)\not\in \{-p,-2p,-3p,\ldots \}$ for all rank-one discrete valuations $\nu$ of $E$ with $\nu|_K = \nu_{\lambda}$ for some $\lambda\in \bar{\mathbb{F}}_p^*$;
\item given a finite Galois extension $E$ of $K$, we let $\mathcal{V}_{E,+}\subseteq E$ denote the set of elements $a\in E$ such that $\nu(a)\ge 0$ for all discrete valuations $\nu$ of $E$ with $\nu|_K = \nu_{\lambda}$ for some $\lambda\in \bar{\mathbb{F}}_p^*$;
\item we let $C$ denote the smallest non-trivial $\bar{\mathbb{F}}_p$-subalgebra of $\bar{\mathbb{F}}_p[[t]]$ that is Artin-Schreier closed in the power series ring $\bar{\mathbb{F}}_p[[t]]$;
\item we let $\widetilde{C}$ denote the smallest non-trivial $\bar{\mathbb{F}}_p$-subalgebra of $\bar{\mathbb{F}}_p((t^{\mathbb{Q}}))$ that is Artin-Schreier closed in the generalized Laurent series ring $\bar{\mathbb{F}}_p((t^{\mathbb{Q}}))$;
\item we let $B$ denote the collection of generalized power series $G(t)$ such that for some $j\ge 0$, $G(t^{p^j})\in L_0$ and is integral over $R$;
\item we let $A$ denote the ring of sparse algebraic power series and we let $\widetilde{A}$ denote the ring of sparse algebraic generalized Laurent series. 
\end{enumerate}
\end{notation}

In terms of generalized series, we have the following more general version of Theorem \ref{thm:main}.

\begin{theorem} \label{thm:main2} Let $p$ be prime and adopt the notation of Notation \ref{not}. Then $$\widetilde{A}=B=\widetilde{C}.$$\end{theorem}

In terms of the above notation, Theorem \ref{thm:main} can be stated as $A=B\cap \bar{\mathbb{F}}_p[[t]]=C$, and so Theorem \ref{thm:main2} is an extension of Theorem \ref{thm:main} in the setting of generalized Laurent series.

In this section, we prove the equalities $A=C$ and $\widetilde{A}=\widetilde{C}$. Proposition \ref{propsparse} shows that $A$ is Artin-Schreier closed in $\bar{\mathbb{F}}_p[[t]]$ and that $\widetilde{A}$ is Artin-Schreier closed in $\bar{\mathbb{F}}_p((t^{\mathbb{Q}}))$.  In particular, we already have shown that we have the containments $C\subseteq A$ and $\widetilde{C}\subseteq \widetilde{A}$. Thus the main content of this section is to prove the reverse inclusion.  The key result in this direction is the following lemma.

\begin{lemma}
\label{lem: gap}
Adopt the notation of Notation \ref{not} and let $d\ge 1$.  Then we have the following:
\begin{enumerate}
\item 
if $F(t)\in C$ and $F(0)=0$ then $F(t)+F(t)^{p^d}+F(t)^{p^{2d}}+\cdots \in C$;
\item if $F(t)\in \widetilde{C}\cap \bar{\mathbb{F}}_p[[t^{\mathbb{Q}}]]_{>0}$ then $F(t)+F(t)^{p^d}+F(t)^{p^{2d}}+\cdots \in \widetilde{C}$ and if $F(t)\in \widetilde{C}\cap \bar{\mathbb{F}}_p((t^{\mathbb{Q}}))_{<0}$ then $F(t)^{p^{-d}}+F(t)^{p^{-2d}}+F(t)^{p^{-3d}}+\cdots \in \widetilde{C}$.
\end{enumerate}
\end{lemma}
\begin{proof}
We only prove part (i), with the proof of part (ii) being handled similarly.
Let $d\ge 1$ and let $x_1,\ldots ,x_d$ be commuting indeterminates.  We let
$A(x_1,\ldots ,x_d)\in M_d(\mathbb{F}_p[x_1,\ldots ,x_d])$ be the matrix whose $(i,j)$-entry is $x_i^{p^{j-1}}$.
Then $f:=\det(A(x_1,\ldots ,x_d))$ is a homogeneous polynomial in the variables $x_1,\ldots ,x_d$ of total degree $p^{d-1}+\cdots + p+1$ with the property that the coefficient of 
$$\prod_{i=1}^d x_i^{p^{i-1}}$$ is nonzero.  By Alon's combinatorial Nullstellensatz \cite[Theorem 1.2]{Alon}, there is some $(a_1,\ldots ,a_d)\in \bar{\mathbb{F}}_p^d$ such that $\det(A(a_1,\ldots ,a_d))\neq 0$, since 
$p^d$ is greater than the degree of $f$. Now let $B=A(a_1,\ldots ,a_d)$. By construction $\det(B)$ is nonzero, and so there is some 
$(c_1,\ldots ,c_d)\in \bar{\mathbb{F}}_p^{1\times d}$ such that 
$(c_1,\ldots ,c_d)B = (1,0,0,\ldots ,0)$.  In other words,
$\sum_{i=1}^d c_i a_i^{p^j} = \delta_{j,0}$ for $j=0,1,\ldots ,d-1$.  Moreover, since $a_i^{p^d}=a_i$ for $i=1,\ldots ,d$, we in fact have
\[ \sum_{i=1}^d c_i a_i^{p^j} = \left\{ \begin{array}{ll} 1 & {\rm if~}j\equiv 0~(\bmod \, d), \\
0 & {\rm otherwise}. \end{array} \right. \]
For $i=1,\ldots ,d$, we let $H_i(t) = a_i F(t) + a_i^p F(t)^p + \cdots $.  Then since $a_i F(t)\in C$ and $C$ is Artin-Schreier closed, we have $H_i(t)\in C$ for $i=1,\ldots ,d$.  Thus
$$\sum_{i=1}^d c_i H_i(t) = F(t)+F(t)^{p^d}+F(t)^{p^{2d}}+\cdots \in C.$$
The proof of part (ii) is done similarly, using the fact that $a_i^{1/p}= a_i^{p^{d-1}}$.
Hence we obtain the desired result.
\end{proof}

We next require a lemma concerning power series whose support set is a simple sparse set.
\begin{lemma} Let $p$ be prime and adopt the notation of Notation \ref{not}. Then the following hold:
\label{lemsparse}
\begin{enumerate}
\item if $S\subseteq \mathbb{N}$ is a non-empty simple sparse subset of $\mathbb{N}$, then $$G(t):=\sum_{n\in S} t^n ~~~{\it is~ in} ~C,$$
\item if $S\subseteq S_p$ is a non-empty well-ordered simple sparse subset of $S_p$, then $$G(t):=\sum_{\alpha \in S} t^{\alpha}~~~{\it is~ in} ~\widetilde{C}.$$
\end{enumerate}
\end{lemma}
\begin{proof}
We first give the proof of (i).  By Remark \ref{rem:sparse}, there is some $s\ge 0$ and some $c_0,\ldots ,c_s\in\mathbb{Z}_{(p)}$ and positive integers $\delta_1,\ldots ,\delta_s$ such that
\[ S = \left\{ c_0 + c_1 p^{\delta_s n_s} + \cdots + c_s p^{\delta_1 n_1+\cdots + \delta_s n_s}\colon n_1, \dots, n_s \ge 0 \right\}.\] Moreover, we have $n\ge c_0$ for all $n\in S$ and $c_0\in S$ if and only if $s=0$.
We prove that $G(t)\in C$ by induction on $s$.  When $s=0$, $G(t)$ is a monomial and the result is clear. Thus we suppose that the result holds whenever $s<m$ with $m\ge 1$ and we consider the case when $s=m$.

Then we may assume that $S$ is a set of natural numbers of the form $$\left\{c_0+c_1 p^{\delta_m n_m} + \cdots + c_m p^{\delta_1 n_1+\cdots + \delta_m n_m}\colon n_1,\ldots ,n_m\ge 0\right\}.$$  We pick a positive integer $N$ that is coprime to $p$ such that $c_i N\in \mathbb{Z}$ for $i=0,\ldots ,m$.
We let
\[ T = \left\{ Nc_1  + Nc_2 p^{\delta_{m-1} n_{m-1}} + \cdots + Nc_m p^{\delta_1 n_1+\cdots + \delta_{m-1} n_{m-1}} \colon n_1, \dots, n_{m-1} \ge 0 \right\}.\]
Then $T$ is a subset of the integers and since $m>1$, every $n\in S$ is strictly greater than $c_0$ and so $T$ is a sparse subset of the positive integers.  We let $H(t)=\sum_{n\in T} t^n$.  Then by the induction hypothesis $H(t)\in C$.
We have $$G(t^N) = \sum_{n\in N\cdot S} t^n = t^{N c_0} \sum_{j\ge 0} \left( \sum_{n\in T} t^n\right)^{p^{j\delta_m}}.$$
That is, $G(t^N) = t^{Nc_0} \left( \sum_{j\ge 0} H(t)^{p^{j\delta_m}}\right)$.  By Lemma \ref{lem: gap}, 
$$ \sum_{j\ge 0} H(t)^{p^{j\delta_m}} \in C.$$  Since $G(t^N)$ is a power series and $C$ is Artin-Schreier closed in $\bar{\mathbb{F}}_p[[t]]$, it follows that $$G(t^N ) =  t^{Nc_0} \left( \sum_{j\ge 0} H(t)^{p^{j\delta_m}}\right)\in C.$$  Since $G(t)$ is also a power series and $C$ is Artin-Schreier closed in $\bar{\mathbb{F}}_p[[t]]$, we have $G(t)\in C$.  The result follows.

The proof of (ii) is handled in a similar manner. We use Remark \ref{rem39} to show that if $S\subseteq S_p$ is a non-empty well-ordered simple sparse set of the form
\[ \begin{split}
\{c_0 &+ c_1 p^{\delta_{j-1} n_{j-1}} + c_2 p^{\delta_{j-1} n_{j-1} + \delta_{j-2} n_{j-2}}\cdots + c_{j-1} p^{\delta_{j-1} n_{j-1}+\cdots +\delta_1 n_1} \\
&+ d_{j-1} + d_j p^{-\delta_j n_j} + d_{j+1} p^{-(\delta_j n_j + \delta_{j+1} n_{j+1})} + \cdots+ d_{s} p^{-(\delta_j n_j + \cdots + \delta_s n_s)} \colon n_1,\ldots ,n_s\ge 0 \}
\end{split}\]
then we have $S=\{d_{j-1}\}+S_1+S_2$ where $S_1$ is the set
$$
\left\{c_0 + c_1 p^{\delta_{j-1} n_{j-1}} + c_2 p^{\delta_{j-1} n_{j-1} + \delta_{j-2} n_{j-2}}\cdots + c_{j-1} p^{\delta_{j-1} n_{j-1}+\cdots +\delta_1 n_1} \colon n_1, \dots, n_{j-1} \ge 0\right\}$$
and $S_2$ is the set
$$\left\{d_j p^{-\delta_j n_j} + d_{j+1} p^{-(\delta_j n_j + \delta_{j+1} n_{j+1})} + \cdots+ d_{s} p^{-(\delta_j n_j + \cdots + \delta_s n_s)} \colon n_j,\ldots ,n_s\ge 0 \right\}.$$
Then $G(t):=\sum_{\alpha \in S} t^{\alpha}$ can be written as a product $t^{d_{j-1}}G_1(t)G_2(t)$, where 
$G_i(t) = \sum_{\alpha\in S_i} t^{\alpha}$ for $i=1,2$.  Then from the above we have $G_1(t)\in C \subseteq \widetilde{C}$ and since $S_2\subseteq (-\infty,0)$, a variant of the above argument used with negative powers of $p$ and applying Lemma \ref{lem: gap} gives that $G_2(t)\in \widetilde{C}$ and so $G(t)\in \widetilde{C}$.
\end{proof}

\begin{proof}[Proof of Theorem \ref{thm:main} (a) and the equality $\widetilde{A}=\widetilde{C}$ in Theorem \ref{thm:main2}]
Proposition \ref{propsparse} gives 
$C\subseteq A$ and $\widetilde{C}\subseteq \widetilde{A}$. We want to show that $A \subseteq C$ and 
$\widetilde{A}\subseteq \widetilde{C}$. We only show that $A\subseteq C$, with the containment $\widetilde{A}\subseteq \widetilde{C}$ being handled in a similar way. 

Let $G(t) = \sum_{n=0}^{\infty} g(n) t^n \in \bar{\mathbb{F}}_p[[t]]$ be an algebraic power series with sparse support. Since $G(t)$ is algebraic, there exists a power $q$ of $p$ such that 
$G(t)\in \mathbb{F}_q[[t]]$ by Theorem \ref{thmK}. For $\alpha \in \mathbb{F}_q^*$, we define $S_{\alpha} := \{ n \in \mathbb{N} \colon g(n) = \alpha\} \subseteq \mathbb{N}$.  By assumption, $S_{\alpha}$ is sparse for each nonzero $\alpha$ in $\mathbb{F}_q$. Then we can write 
$$G(t) = \sum_{\alpha \in \mathbb{F}_q^*} \alpha \left( \sum_{n \in S_{\alpha}} t^n \right).$$ 
Since each $S_{\alpha}$ is sparse, by Equation (\ref{eq:sqcup}), each $S_{\alpha}$ admits a decomposition into disjoint sets
$$\bigsqcup_{i=1}^{d_{\alpha}} S_{\alpha,i}$$ for some integer $d_{\alpha} \geq 1$ with each $S_{\alpha,i}$ a simple sparse set. For $\alpha\in \mathbb{F}_q$ and $i=1,\ldots ,d_{\alpha}$, we define
$$G_{S_{\alpha,i}}(t):=\sum_{n \in S_{\alpha,i}} t^n.$$ Then we have
$$G(t) = \sum_{\alpha \in \mathbb{F}_q^*} \alpha \left( \sum_{i=1}^{d_{\alpha}} G_{S_{\alpha,i}}(t)\right).$$  
Now by Lemma \ref{lemsparse}, each $G_{S_{\alpha,i}}(t)$ is in $C$ and so $G(t)$ is also in $C$.  The result follows.
\end{proof}

\section{Proof of Theorem \ref{thm:main} ($b$)} \label{mainresb}
We now prove Theorem \ref{thm:main} (b), which in terms of Notation \ref{not} can be expressed as 
$A=B\cap \bar{\mathbb{F}}_p[[t]]$. In order to prove this equality, we must first obtain a description of $L_0$, which appears in the definition of $B$.

 To give a better picture of $L_0$, it is necessary to first know all valuations of $K$.  We recall that the places of the field $\bar{\mathbb{F}}_p(t)$ that are constant on $\bar{\mathbb{F}}_p$ are parametrized by the projective line over $\bar{\mathbb{F}}_p$ (see Zariski-Samuel \cite[Ch. VI, \S17]{ZS}).  Since valuations of $\bar{\mathbb{F}}_p$ are all trivial, these are in fact all places.  Each such place is a discrete valuation of $\bar{\mathbb{F}}_p(t)$ and for $\lambda\in \mathbb{P}^1\setminus \{0,\infty\}$ these valuations lift uniquely to a valuation of $K$ with the same value group as its restriction to $\bar{\mathbb{F}}_p(t)$; that is, $K$ is an extension of $\bar{\mathbb{F}}_p(t)$ that is unramified outside of $0$ and $\infty$. If we had included $p$-th power roots of $t$ in our definition of $K$, the value groups of the extensions of these valuations would necessarily increase from $\mathbb{Z}$ to $\mathbb{Z}[1/p]$.
 
We begin with a simple remark characterizing integral closure in terms of valuations that we shall use in the proof of the main theorem.  
\begin{remark} \label{val} Adopt the notation of Notation \ref{not} and let $E$ be a finite Galois extension of $K$.  Then $a\in \mathcal{V}_{E,+}$ if and only if $a$ is integral over $R$.
\end{remark}
\begin{proof} First suppose that $a\in E$ is integral over $R$. Then $a$ satisfies a non-trivial polynomial equation $a^n + r_{n-1} a^{n-1}+\cdots + r_0=0$ for some $n\ge 1$ and $r_0,\ldots ,r_{n-1}\in R$.  Then if $\mu$ is a valuation of $E$ with $\mu|_K=\nu_{\lambda}$ for some $\lambda\in \bar{\mathbb{F}}_p^*$ then $\nu(r_i)\ge 0$ for $i=0,\ldots ,n-1$. Then if $\nu(a)<0$ for some $\nu\in \mathcal{X}$, we necessarily have $\nu(a^n) = n\nu(a) < i \nu(a)\le  \nu(r_i a^i) $ for $i=0,\ldots ,n-1$, which contradicts the fact that $a^n = - (r_{n-1} a^{n-1}+\cdots + r_0)$.  Thus $a\in \mathcal{V}_{E,+}$. 
Conversely, suppose that $a\in \mathcal{V}_{E,+}$ and that $a$ is not integral over $R$.  Then since $a$ is not integral over $R$ and $a$ is necessarily nonzero we have $a$ is not integral over $R[a^{-1}]$ since otherwise, we'd have a non-trivial polynomial relation of the form $0=a^n + p_{n-1}(a^{-1}) a^{n-1} + \cdots + p_0(a^{-1})$, with each $p_i(a^{-1})\in R[a^{-1}]$, and then multiplying by a sufficiently large power of $a$ would give that $a$ is integral over $R$.  In particular, $a^{-1}$ is not a unit of the integral closure $S$ of $R[a^{-1}]$ and so there is a height one prime $Q$ of $S$ such that $a^{-1}\in Q$.   Then the local ring $S_Q$ is a discrete valuation ring and the valuation $\nu$ on $E$ induced by $Q$ gives a rank-one discrete valuation of $K$ corresponding to the valuation induced by the prime ideal $R\cap Q$ of $R$.  In particular, there is some $\lambda\in \bar{\mathbb{F}}_p^*$ such that $\nu|_K$ is equivalent to $\nu_{\lambda}$. Now by construction $a^{-1}\in Q$ and so $\nu(a^{-1})>0$ and thus $\nu(a)<0$, which contradicts the fact that $a\in \mathcal{V}_{E,+}$.  
\end{proof}
 \begin{lemma}
 Adopt the notation from Notation \ref{not}, let $E$ be a finite extension of $K$, let $\lambda\in \bar{\mathbb{F}}_p^*$, and let $\mathcal{Y}$ be the set of valuations of $E$ whose restriction to $K$ is equal to $\nu_{\lambda}$.  Then for each $\mu\in \mathcal{Y}$, there exists $\epsilon\in \mathcal{V}_{E,+}$ such that $\mu'(\epsilon-\delta_{\mu,\mu'})>0$ for all $\mu'\in \mathcal{Y}$.
 \label{CRT}
 \end{lemma}
 \begin{proof} Let $T$ denote the integral closure of $R$ in $E$.  Then $P:=\{r\in R\colon \nu_{\lambda}(r)>0\}$ is a maximal ideal of $R$. Let $Q_1,\ldots ,Q_s$ denote the prime ideals of $T$ that lie above $R$.  Then each local ring $T_{Q_i}$ is a discrete valuation ring and each $\mu\in Y$ is induced by one of these valuation rings. Since the $Q_i$ are maximal ideals, they are in particular pairwise comaximal, and so we see by the Chinese remainder theorem there exists some $\epsilon_i\in T$ such that $\epsilon_i-\delta_{i,j} \in Q_j$.  The fact that each $\epsilon_i\in \mathcal{V}_{E,+}$ follows from Remark \ref{val}. The result follows.
 \end{proof} 
 
\begin{lemma} \label{lem: VL} Adopt the notation from Notation \ref{not}, let $E$ be a Galois extension of $K$ of degree $p^m$ for some $m\ge 0$ that is unramified outside of $0$ and $\infty$, and let $a$ be a nonzero element of $E$.  Then there is some $b\in E$ such that $a-(b^p-b)$ is in $\mathcal{V}_{E}$.
\end{lemma}
\begin{proof}
Let $\mathcal{X}$ denote the set of valuations on $E$ whose restriction to $K$ is of the form $\nu_{\lambda}$ for some $\lambda\in \bar{\mathbb{F}}_p^*$. Then there are finitely many $\mu\in \mathcal{X}$ such that $\mu(a)<0$. Since $E$ is an extension of $K$ that is unramified outside of $0$ and $\infty$, the value group of each $\mu\in \mathcal{X}$ is the same as the value group of the corresponding $\nu_{\lambda}$ and so $(t-\lambda)$ is a uniformizing parameter for the valuation ring of $\mu$. Let
$\mu_1,\ldots ,\mu_d$ be the finite set of valuations in $\mathcal{X}$ for which $\mu_i(a) \in \{-p, -2p, \dots\}$ for $i \in \{1, \dots, d\}$, and let $m_1,\ldots ,m_d$ be the positive integers such that $\mu_i(a)=-p m_i$.  

Let $M = M(a) := m_1 + \cdots + m_d$. We prove the claim by induction on $M$. When $M=0$ (i.e., $d=0$), there is nothing to show. We next assume that the claim holds whenever $M<N$ and we consider the case when $M=m_1+\cdots +m_d=N$.  Let $\lambda\in  \bar{\mathbb{F}}_p^*$ be such that 
$\mu_1|_K = \nu_{\lambda}$. Then since $\mu_{1}(a)=-m_1 p$, there is some $c\in \bar{\mathbb{F}}_p^*$ such that $\mu_1(a - c^p/(t-\lambda)^{m_1p}) > \mu_1(a)$.
By Lemma \ref{CRT}, there is some $\epsilon\in \mathcal{V}_{E,+}$ such that $\mu(\epsilon -\delta_{\mu,\mu_1}) > 0$ for all valuations $\mu\in\mathcal{X}$ with $\mu|_{K} = \mu_1|_K$.  It follows that for a sufficiently large $s>1$ we have
$\mu(a- c^p \epsilon^{p^s}/(t-\lambda)^{m_1p} + c \epsilon^{p^{s-1}}/(t-\lambda)^{m_1}) = \mu(a)$ for $\mu\in \mathcal{X}$ and $\mu\neq \mu_1$ and $\mu_1(a- c^p \epsilon^{p^s}/(t-\lambda)^{m_1p} + c \epsilon^{p^{s-1}}/(t-\lambda)^{m_1} ) < \mu_1(a)$.  Letting $a' = a - c^p \epsilon^{p^s}/(t-\lambda)^{m_1p} + c \epsilon^{p^{s-1}}/(t-\lambda)^{m_1}$ and noting that $\epsilon\in \mathcal{V}_{E,+}$, we see that $M(a') < M(a)$.  Then by the induction hypothesis there is some $b\in E$ such that $a'-b^p + b\in \mathcal{V}_{E}$ and so $a - (b')^p + b'\in \mathcal{V}_{E}$ with 
$$b' =  b + c \epsilon^{p^{s-1}}/(t-\lambda)^{m_1}.$$  The result follows.
\end{proof}

\begin{lemma} Adopt the notation of Notation \ref{not}. Then $\widetilde{C}\subseteq B$ and $C\subseteq B\cap \bar{\mathbb{F}}_p[[t]]$. \label{lem:BC}
\end{lemma}
\begin{proof} By the definitions of $\widetilde{C}$ and $C$ it suffices to show that $B$ is Artin-Schreier closed in the ring of generalized Laurent series and that $B \cap  \bar{\mathbb{F}}_p[[t]]$ is Artin-Schreier closed in the ring of formal power series over $\bar{\mathbb{F}}_p$, since $B$ contains $R$. 

Since the substitutions $t\mapsto \alpha t$ with $\alpha\in \bar{\mathbb{F}}_p^*$ preserve $\mathbb{P}^1\setminus \{0,\infty\}$, $B$ and $B\cap \bar{\mathbb{F}}_p[[t]]$ both have property (P2) in Definition \ref{def:AS}.   Similarly, since for $c$ a nonzero rational number with $p$-adic valuation $1$, the map $t\mapsto t^c$ induces an automorphism of $R$; moreover, if $c>0$ then this extends to an automorphism of the ring of generalized Laurent series. Thus if $F(t)\in B$ then $F(t^c)\in B$ and if $F(t)\in B\cap  \bar{\mathbb{F}}_p[[t]]$ and $F(t^c)\in  \bar{\mathbb{F}}_p[[t]]$ then $F(t^c)\in B \cap \bar{\mathbb{F}}_p[[t]]$. By construction, $B$ is closed under the maps induced by $t\mapsto t^p$ and $t\mapsto t^{1/p}$ and so we then see property (P3) holds in both cases.  

Thus it remains to verify that property (P1) holds.  
Suppose that $F(t)\in B$. Then there is some $j\ge 0$ such that $G(t):=F(t^{p^j})$ is in $L_0$ and is integral over $R$. Let $H(t)$ be a generalized Laurent series that is a solution to $X^p-X=G(t)$.  Since $G(t)$ is integral over $R$, $H(t)$ is integral over $R$.  Let $E$ denote an extension of $K$ containing $G$ that is unramified outside of $0$ and $\infty$. Since $G(t)$ is integral over $R$, we have $G\in \mathcal{V}_{E,+}$ by Remark \ref{val}. Let $\nu$ be valuation of $E$ whose restriction to $K$ corresponds to $\nu_{\lambda}$ with $\lambda\in \bar{\mathbb{F}}_p$. Then since $E$ is an extension of $K$ that is unramified outside of $0$ and $\infty$, and $\nu(G)\ge 0$, we can complete with respect to the valuation $\nu$ and we see that $G$ has a formal power series expansion in the variable $u=t-\lambda$.  Write
$G(u) =\sum_{i\ge 0} a_i u^i$ and let $G_{+}(u)=\sum_{i\ge 1} a_i u^i$.  Then $H_1(u):=G_{+}(u)+G_{+}(u^p)+\cdots $ is a power series in $u$ and since $H$ is a solution to $X^p-X=G$, $H$ is of the form
$H_1+ \beta$ for some $\beta\in \bar{\mathbb{F}}_p$ with $\beta^p-\beta = a_0$.  In particular, $H$ lies in the completion of the valuation ring of $\nu$ and so the value groups of the extensions of $\nu$ to $E(H)$ are all equal to the value group of $\nu$ (more precisely $u$ is a uniformizing parameter for the valuation ring of $E(H)$ for each valuation above $\nu$) and so $E(H)$ is unramified at all extensions of $\nu$ to $E(H)$. Thus the field $E(H)$ is an extension of $K$ that is unramified outside of $0$ and $\infty$ and so $H\in B$.  The result follows.  Thus $B$ has properties (P1)--(P3). The result follows.
 \end{proof}
 
 \begin{proof}[Proof of Theorem \ref{thm:main} (b) and the equality $B=\widetilde{A}$ in Theorem \ref{thm:main2}] We adopt the notation of Notation \ref{not}.
Suppose that $G(t)\in B$.  We shall first show that $G(t)$ is sparse.  We may replace $G(t)$ by $G(t^{p^j})$ and assume that $G\in L_0$.  Let $E$ denote the Galois closure of $G(t)$ over $K$.  Then $[E:K]=p^m$ for some $m\ge 0$.  We prove this by induction on $m$.  If $m=0$, $G\in K$ and since $G$ is integral over $R$ and $R$ is integrally closed in $K$, $G\in R$, and so $G$ is easily seen to be sparse in this case, as elements of $R$ have finite support. Now we suppose that the result holds whenever $m<s$ with $s\ge 1$ and we consider the case when $m=s$.
 
Then by Remark \ref{rem:Galois}, there is a Galois extension $E_0$ of $K$ of degree $p^{s-1}$ that is unramified outside of $0$ and $\infty$ with $E=E_0[H]$ and $H^p-H=F\in E_0$.  By Lemma \ref{lem: VL}, there is some $b\in E_0$ such that $F-b^p+b\in \mathcal{V}_{E_0}$.  Thus after replacing $H$ by $H+b$, we may assume that $F\in \mathcal{V}_{E_0}$.  
Then notice in fact we must have $F\in \mathcal{V}_{E_0,+}$ since otherwise there would be some valuation $\nu$ of $E_0$ with $\nu(F)<0$ and $p\nmid \nu(F)$. Then since $H^p-H=F$, we have $\nu'(H)=\nu(F)/p\not\in\mathbb{Z}$ for every extension $\nu'$ of $\nu$ to $E$. In particular, this contradicts the fact that $E$ is unramified outside of $0$ and $\infty$. Thus $F$ is integral over $R$ by Remark \ref{val} and so $F\in B$ and thus $F$ is sparse by the induction hypothesis.  Hence $H$ is a sparse generalized power series by Proposition \ref{propsparse}.
Moreover, $H\in B$ as $E$ is unramified over $K$ outside of $0$ and $\infty$ and $H$ is integral over $R$ since $F$ is integral over $R$. Then since $G(t)\in E_0[H]$, we can write $G= e_0+e_1 H +\cdots + e_{p-1} H^{p-1}$ with $e_0,\ldots ,e_{p-1}\in E_0$.  We claim that $e_0,\ldots ,e_{p-1}\in B$. To see this, suppose that this is not the case. Then there is some largest $i\ge 0$ for which $e_i\not\in B$.  Then $G_0=e_0+e_1 H +\cdots + e_i H^i\in B$, since $G\in H$ and $e_j H^j\in B$ for $j>i$.  Since $E$ is a Galois extension of $E_0$ and $H^p-H\in E_0$, there is an automorphism of $E$ that fixes $E_0$ element-wise and sends $H$ to $H+1$.  Since $\sigma$ fixes $K$ element-wise, $\sigma$ preserves elements of $B$.  In particular, the operator $\Delta: E\to E$ given by $\Delta(a)=\sigma(a)-a$ maps elements of $B\cap E$ to elements of $B\cap E$.  Notice that
$$\Delta^i(G_0) = i! e_i\not\in B,$$ which is a contradiction, since $G_0\in B$ and $B$ is preserved under application of $\Delta$.  Thus $e_0,\ldots ,e_{p-1}\in B$.  By the induction hypothesis, $e_0,\ldots ,e_{p-1}$ are sparse generalized power series and since $H$ is also a sparse generalized power series, $G$ must be too, since sparse series form a ring. The result now follows by induction.  Hence $G(t)\in \widetilde{A}$.  By Lemma \ref{lem:BC} and the fact that $\widetilde{A}=\widetilde{C}$, established earlier, we see that $\widetilde{A}=\widetilde{C}\subseteq B$ and so $B=\widetilde{A}$.  It is straightforward to show that $\widetilde{A}\cap \bar{\mathbb{F}}_p[[t]] = A$ and so we also get $B\cap \bar{\mathbb{F}}_p[[t]] = A$.  This completes the proof.
\end{proof}
\section*{Acknowledgment} 
We are indebted to Jakub Byszewski for correcting an earlier statement of Theorem \ref{thm:main} and suggesting a suitable reformulation.
 
\end{document}